\documentclass[11pt,a4paper]{amsart}
\usepackage[english]{babel}
\usepackage[T1]{fontenc}
\usepackage{mathpazo,amssymb,bm}
\usepackage{upref}
\usepackage{enumerate}
\usepackage{graphicx,xcolor}
\usepackage{cancel}
\usepackage{dsfont}       
\usepackage[colorlinks=true, pdftitle={Regularizing effect of homogeneous evolution
equations with perturbation}, pdfauthor={Daniel Hauer} ]{hyperref}

\title[Homogeneous evolution equations with perturbation]{Regularizing
effect of homogeneous evolution equations with perturbation}

\author{Daniel Hauer} \address[Daniel Hauer]{School of Mathematics and
  Statistics, The University of Sydney, NSW 2006, Australia}
\email{\href{mailto:daniel.hauer@sydney.edu.au}{\nolinkurl{daniel.hauer@sydney.edu.au}}}

\thanks{The author's research was supported by an Australian Research
  Council grant DP200101065.}



\numberwithin{equation}{section}

\newtheorem{theorem}{Theorem}[section]
\newtheorem{proposition}[theorem]{Proposition}
\newtheorem{lemma}[theorem]{Lemma}
\newtheorem{corollary}[theorem]{Corollary}

\theoremstyle{definition}
\newtheorem{definition}[theorem]{Definition}

\newtheorem{remark}[theorem]{Remark}

\newcommand\R{{\mathbb{R}}}
\newcommand\N{\mathbb{N}}

\newcommand\E{\mathcal{E}}

\newcommand\dx{\mathrm{d}x }
\newcommand\dy{\mathrm{d}y }
\newcommand\dr{\mathrm{d}r }
\newcommand\ds{\mathrm{d}s }

\newcommand\dmu{\mathrm{d}\mu}
\newcommand\dt{\mathrm{d}t }
\newcommand\td{\mathrm{d} }

\DeclareMathOperator*{\divi}{div\!\!}

\DeclareMathOperator{\Var}{Var}

\newcommand\abs[1]{\lvert#1\rvert}
\newcommand\labs[1]{\left\lvert#1\right\rvert}
\newcommand\norm[1]{\lVert#1\rVert}
\newcommand\lnorm[1]{\left\lVert#1\right\rVert}

\definecolor{darkred}{rgb}{0.7,0.1,0.1}

\def\divi{\hbox{\rm div\,}}

\begin{document}
\date{\today}
\maketitle

\tableofcontents

\begin{abstract}
  Since the pioneering works by
  Aronson  \& B\'enilan [C. R. Acad. Sci. Paris S\'er., 1979],
  and B\'enilan \& Crandall [Johns Hopkins Univ. Press, 1981]
  it is well-known that first-order evolution problems
  governed by a nonlinear but homogeneous operator admit the smoothing
  effect that every corresponding mild solution is Lipschitz
  continuous at every positive time. Moreover, if the underlying Banach
  space has the Radon-Nikod\'ym property, then these mild solution is
  a.e. differentiable, and the time-derivative 
  satisfies glo\-bal and point-wise bounds.
  
  In this paper, we show that these results remain true if the
  homogeneous operator is perturbed by a Lipschitz continuous
  mapping. More precisely, we establish global \emph{$L^1$
    Aronson-B\'enilan type estimates} and \emph{point-wise
    Aronson-B\'enilan type estimates}. We apply our theory to derive
  global $L^q$-$L^{\infty}$-estimates on the time-derivative of the
  perturbed diffusion problem governed by the Dirichlet-to-Neumann operator
  associated with the $p$-Laplace-Beltrami operator and lower-order
  terms on a compact
  Riemannian manifold with a Lipschitz boundary. 
\end{abstract}

%
%

\section{Introduction and main results}

In this paper, we establish global regularity estimates on the
time-derivative $\frac{\td u}{dt}$ of \emph{mild} solutions $u$ (see
Definition~\ref{def:mild-solution}) to the Cauchy problem associated
with the perturbed operator $A+F$;
\begin{equation}
  \label{eq:10bisF}
  \begin{cases}
    \frac{\td u}{\dt}+A(u(t))+ F(u(t))\ni f(t) & \text{for
      $t\in (0,T)$,}\\
    \phantom{\frac{\td u}{\dt}+A(u(t))+ F(}u(0)=u_{0},&
  \end{cases}
\end{equation}
for sufficiently regular $f : [0,T]\to X$ and initial data $u_{0}$.
To ensure the well-posedness of Cauchy problem~\eqref{eq:10bisF}, we
assume that $A$ is an $m$-accretive, possibly, multi-valued operator
$A : D(A)\to 2^{X}$ on a Banach space $(X,\norm{\cdot}_{X})$ (see
Definition~\ref{def:quasi-accretive}) with \emph{effective domain}
$D(A) :=\{u\in X\,\vert\,Au\neq \emptyset\}$ and $F : X\to X$ a
Lipschitz continuous mapping with constant $\omega\ge 0$ satisfying
$F(0)=0$. 

The crucial condition to obtain global regularity estimates of $\frac{\td u}{dt}$
for mild solutions $u$ of~\eqref{eq:10bisF} is that $A$ is
\emph{homogeneous} of order $\alpha\neq 1$; that is, $(0,0)\in A$ and
\begin{equation}
  \label{eq:41}
  A(\lambda u)=\lambda^{\alpha}Au\qquad\text{ for all $\lambda\ge 0$
and $u\in D(A)$.}
\end{equation}
We emphasize that the governing operator $A+F$ in Cauchy
problem~\eqref{eq:10bisF} is not anymore homogeneous. Thus, our first
main result can be understood as a perturbation theorem.

\begin{theorem}[{$L^1$ Aronson-B\'enilan type estimates}]\label{thm:main1}
  For given $\alpha\in \R\setminus\{1\}$, let $A$ be an $m$-accretive
  operator in $X$ which is homogeneous of order $\alpha$
  and suppose, the mapping $F : X \to X$ is Lipschitz continuous on
  $X$ with constant $\omega\ge 0$, $F(0)=0$, and let $f\in BV(0,T;X)$.
  Then for every $u_{0}\in D(A)$, the mild solution $u$ of~\eqref{eq:10bisF} satisfies
   \begin{equation}
     \label{eq:28}
    \limsup_{h\to 0+}\frac{\norm{u(t+h)-u(t)}_{X}}{h}
    \le\frac{1}{t}\!\! \left[a_{\omega}(t)+ \omega\!\!\int_{0}^{t}a_{\omega}(s)
      e^{\omega (t-s)}\ds\right]
  \end{equation}
  for a.e. $t\in (0,T)$, where
  \begin{equation}
    \label{eq:54}
  \begin{split}
    a_{\omega}(t)&:= V_{0}(f,t) + \frac{1}{\abs{1-\alpha}}\bigg[\left(1+e^{\omega
    t} \right)\,\norm{u_{0}}_{X}\bigg.\\
  &\hspace{1cm}\left. +\int_{0}^{t}\norm{f(s)}_{X}\,\ds +\omega\,\int_{0}^{t} \int_{0}^{s}e^{-\omega r} \norm{f(r)}_{X}\dr\,\ds\right].  
  \end{split}
\end{equation}
  and $V_{0}(f,\cdot)$ is given by~\eqref{eq:64} below. In particular,
  if for $u_{0}\in D(A)$, the right-hand side derivative $\frac{\td
  u}{\dt}_{\!\! +}$ exists, then 
    \begin{equation}
     \label{eq:29}
     \lnorm{\frac{\td u}{\dt}_{\!\! +}\!\!(t)}_{X}
    \le\frac{1}{t}\!\! \left[a(t)+ \omega\!\!\int_{0}^{t}a(s)
      e^{\omega (t-s)}\ds\right]\quad\text{for a.e. $t\in (0,T)$.}
  \end{equation}
\end{theorem}

At the first view, it seems that in Theorem~\ref{thm:main1}, the 
hypothesis $u_{0}\in D(A)$ merely provides a global point-wise estimate
on the time-derivative $\frac{\td u}{\dt}(t)$, but not a
regularization effect. This hypothesis together with the condition
$f\in BV(0,T;X)$ imply that the mild solution $u$ is
Lipschitz continuous (see Proposition~\ref{prop:Lipschitz-property}),
which is required to apply Gronwall's lemma (see Lemma~\ref{lem:1}. But,
starting from this, a standard density argument combined with an
appropriate compactness result yield that estimate~\eqref{eq:28} holds
for all mild solutions $u$ of Cauchy problem~\eqref{eq:10bisF}.

For example, under the additional hypothesis that the Banach space $X$
is \emph{reflexive}, one has that the closed unit ball of $X$ is
weakly sequentially compact. Now, for every given
$u_{0}\in \overline{D(A)}^{\mbox{}_{X}}$, there is a sequence
$(u_{n}^{(0)})_{n\ge 1}$ in $D(A)$ such that $u_{n}^{(0)}\to u_{0}$ in
$X$ and by the $\omega$-quasi contractivity of the semigroup $\{T_{t}\}_{t=0}^{T}$
generated by $-(A+F)$ on $\overline{D(A)}^{\mbox{}_{X}}\times
L^{1}(0,T;X)$ (see Definition~\ref{def:semigroup}), one has that
$T_{t}(u_{n}^{(0)},f)\to T_{t}(u_{0},f)$ in $X$ as $n\to\infty$. Thus, if for
every $n\ge 1$, $u_{n}(t):=T_{t}(u_{n}^{(0)},f)$, $t\ge 0$,
satisfies~\eqref{eq:29}, then the sequence
$(\frac{\td u_{n}}{\dt})_{n\ge 1}$ is bounded $L^{\infty}(\delta,T;X)$
for every $\delta\in (0,T)$. From this, one can conclude the following
smoothing effect of such semigroups acting on reflexive Banach spaces
(see also Corollary~\ref{cor:Lipschitz-continuous-case} in Section~\ref{sec:semigroups}). 

\begin{corollary}
  \label{cor:RN-Lipschitz-case}
  Let $A$ be an $m$-accretive operator on a reflexive Banach space
  $X$, $F : X\to X$ a Lipschitz continuous mapping with
  Lipschitz-constant $\omega\ge 0$ satisfying $F(0)=0$, and
  $\{T_{t}\}_{t=0}^{T}$ the semigroup generated by $-(A+F)$ on
  $\overline{D(A)}^{\mbox{}_{X}}\times L^{1}(0,T;X)$. If $A$ is
  homogeneous of order $\alpha\neq 1$, then for every
  $u_{0}\in \overline{D(A)}^{\mbox{}_{X}}$ and $f\in BV(0,T;X)$, the
  unique mild solution $u$ of Cauchy
  problem~\eqref{eq:10bisF} is strong and 
  satisfies~\eqref{eq:29} for a.e. $t\in (0,T)$.
\end{corollary}

We outline the proof of this corollary in
Section~\ref{sec:semigroups}.\medskip

Our second main result of this paper is concerned with a point-wise
estimate on the time-derivative $\frac{\td u}{dt}$ of
positive\footnote{That is, $u\ge 0$ for the given partial ordering
  $``\!\!\!\le''$ on $X$.} strong solutions $u$ of the homogeneous Cauchy
problem
\begin{equation}
  \label{eq:10bis-no-F}
  \begin{cases}
    \frac{\td u}{\dt}+A(u(t))+ F(u(t))\ni 0 & \text{for
      $t\in (0,T)$,}\\
    \phantom{\frac{\td u}{\dt}+A(u(t))+ F(}u(0)=u_{0},&
  \end{cases}
\end{equation}
under the additional hypothesis that the underlaying Banach space $X$
is equip\-ped with a partial ordering $``\!\!\!\le''$ such that the
triple $(X,\norm{\cdot}_{X},\le)$ defines an Banach lattice, and if
for this ordering $``\!\!\!\le''$, every mild solution $u$
of~\eqref{eq:10bis-no-F} is \emph{order-preserving}; that is, for every
$u_{0}$, $\hat{u}_{0}\in \overline{D(A)}^{\mbox{}_{X}}$ with
corresponding mild solutions $u$ and $\hat{u}$ of~\eqref{eq:10bis-no-F},
one has that $u_{0}\le \hat{u}_{0}$ implies $u(t)\le \hat{u}(t)$ for
all $t\in (0,T]$.

\begin{theorem}[{Point-wise Aronson-B\'enilan type estimates}]
  \label{thm:2}
  Let $A$ be an $m$-accre\-tive operator on $X$,
  $(X, \norm{\cdot}_{X},\le)$ a Banach lattice, and let $F : X \to X$
  be a Lipschitz continuous mapping on $X$ with constant $\omega\ge 0$
  satisfying $F(0)=0$. Suppose, for $\alpha\in \R\setminus\{1\}$, $A$
  is homogeneous of order $\alpha$ and every mild solution $u$
  of~\eqref{eq:10bis-no-F} is \emph{order-preserving}. For every positive
  $u_{0}\in \overline{D(A)}^{\mbox{}_{X}}$,
  the mild solution $u$ of~\eqref{eq:10bis-no-F} satisfies
  \begin{displaymath}
   \frac{u(t+h)-u(t)}{h}\ge
   \frac{(1+\frac{h}{t})^{\frac{1}{1-\alpha}}-1}{h}\frac{u(t)}{t}+g_{h}(t)
   \qquad\text{if $\alpha>1$}
 \end{displaymath}
 and
 \begin{displaymath}
   \frac{u(t+h)-u(t)}{h}\le
   \frac{(1+\frac{h}{t})^{\frac{1}{1-\alpha}}-1}{h}\frac{u(t)}{t}+g_{h}(t)
   \qquad\text{if $\alpha<1$,}
 \end{displaymath}
 for every $t$, $h>0$, where $g_{h} : (0,\infty)\to X$ is a continuous
 function. Further, for positive $u_{0}\in
 \overline{D(A)}^{\mbox{}_{X}}$, if the right hand-side derivative
 $\frac{\td u}{\dt}_{\!\! +}$ belongs to $L^{1}_{loc}([0,\infty);X)$,
 then
  \begin{equation}
    \label{eq:14}
    (\alpha-1)\frac{\td u}{\dt}_{\!\!+}\!\!(t)\ge -\frac{u(t)}{t}+(\alpha-1)g_{0}(t),
  \end{equation}
  for a.e. $t>0$, where $g_{0} : (0,\infty)\to X$ is a measurable function. 
\end{theorem}

Theorem~\ref{thm:2} follows from the slightly more
general statement provided in Theorem~\ref{thm:2bis} and
by Corollary~\ref{cor:2bis} in Section~\ref{sec:main-results}.

It is worth mentioning some words about the origin of the names
assigned to the estimates~\eqref{eq:28} (respectively,~\eqref{eq:29})
and~\eqref{eq:14}. Even though the result was already mentioned
earlier in~\cite[p. 5]{MR255996} by Aronson, the point-wise
estimate~\eqref{eq:14} was first proved by Aronson \&
B\'enilan~\cite{MR524760} for (strong) solutions $u$ of the porous
medium equation $u_{t}=\Delta u^{m}$ in $[0,+\infty)\times \R^{d}$ for
$d\ge 1$ and $m>[d-2]^{+}/d$. In the same
paper~\cite[Th\'eor\`eme~2.]{MR524760}, they also proved that (strong)
solutions of this porous media equation satisfy the
$L^1$-estimate~\eqref{eq:29}. Shortly afterwards, B\'enilan and
Crandall~\cite{MR648452} made available the two global
inequalities~\eqref{eq:28} and~\eqref{eq:14} for mild solutions $u$ of
the unperturbed Cauchy problem
\begin{equation}
  \label{eq:10}
  \begin{cases}
    &\frac{\td u}{\dt}+A(u(t))\ni 0\qquad \text{in $(0,\infty)$,}\\
   &\;\phantom{\frac{\td u}{\dt}+A} u(0)=u_{0},
  \end{cases}
\end{equation}
governed by nonlinear $m$-accretive operators $A$, which are
homogeneous of order $\alpha>0$, $\alpha\neq 0$. This class of
operators include the local $p$-Laplace operator $\Delta_{p}$, the
local doubly nonlinear operator $\Delta_{p}u^{m}$, $1<p<\infty$,
$m>0$, as well as the nonlocal fractional $p$-Laplace operator
$(-\Delta_{p})^{s}$, respectively equipped with various boundary
conditions (see, for instance,~\cite{CoulHau2016} for more details to
the analytic properties of these quasi-linear 2nd-order differential
operators).

In the papers~\cite{MR647071} and \cite{MR670925} Crandall and Pierre
showed that every mild solution of the more general version of the
porous medium equation $u_{t}=\Delta \varphi(u)$, where $\varphi$ is
an increasing function on $\R$, also satisfy the point-wise
Aronson-B\'enilan estimate~\eqref{eq:14}. These two results by
Crandall and Pierre were slightly improved in a short paper by
Casseigne~\cite{MR1993968}.  Estimate~\eqref{eq:14} has been
established in various settings; on manifolds (see,
e.g.~\cite{MR2487898,MR3394614}), and with drift-term (see,
e.g,~\cite{MR2921651}), or with a linear perturbation (see, e.g.,
\cite{MR3765564}). One important reason among others, for the strong
further development of the point-wise estimate~\eqref{eq:14} is that
it can be used, for example, to derive Harnack-type inequalities~(see,
e.g.,~\cite{MR1219716}, but also~\cite{MR1230384,MR2865434}) and to
study the regularity of the free-boundaries (see, for instance,
\cite{MR659697} or~\cite{MR1260981}). We refer the interested reader
to the book~\cite{MR2286292} by V\'azquez (and more
recently~\cite{bevilacqua2020aronsonbnilan}) for a detailed exposition
concerning the development of the point-wise Aronson-B\'enilan
estimate~\eqref{eq:14} satisfied by solutions to the porous media
equation.

Recently, the author and Maz\'on showed in~\cite{MR4031770} that the
two Aronson-B\'enilan type estimates~\eqref{eq:29} and~\eqref{eq:14}
are satisfied by the mild solutions of the unperturbed Cauchy
problem~\eqref{eq:10} for homogeneous operators of order zero (i.e.,
$\alpha=0$). This class of operators includes, for example, the
(negative) total variational flow operator
$Au=-\textrm{div}(\frac{D u}{\abs{D u}})$, or the $1$-fractional
Laplacian $A=(-\Delta_{1})^{s}$ for $s\in (0,1)$ respectively equipped
with some boundary conditions. By tackling the the $L^1$
Aronson-B\'enilan inequality~\eqref{eq:29} for mild solutions of the
perturbed (homogeneous) Cauchy problem~\eqref{eq:10bis-no-F}, their
proof, unfortunately, contains a slightly wrong argument in the
application of Gronwall's lemma. Thus, the proof of
Theorem~\ref{cor:1bisF} presented here corrects this flaw.

If the operator $A$ in~\eqref{eq:10}
is \emph{linear} (and hence $\alpha=1$), then 
estimate
\begin{equation}
  \label{eq:30}
  \norm{Au(t)}_{X}\le
  C\frac{\norm{u(0)}_{x}}{t},\qquad (t\in (0,1],\; u(0)\in D(A)),
\end{equation}
yields that the operator $-A$ generates an analytic semigroup
$\{T_{t}\}_{t\ge 0}$ (cf.,~\cite{MR2103696,MR710486}). Thus, it is
interesting to see that a regularity inequality~\eqref{eq:29}, which
is similar to~\eqref{eq:30}, also holds for nonlinear operators of the
type $A+F$, where $A$ is homogeneous of order $\alpha\neq 1$. Further,
if the norm $\norm{\cdot}_{X}$ is induced by an inner product
$(\cdot,\cdot)_{X}$ of a Hilbert space $X$ and $A=\partial\varphi$ is
the sub-differential operator $\partial\varphi$ in $X$ of a
semi-convex, proper, lower semicontinuous function
$\varphi : X\to (-\infty,+\infty]$, then regularity
inequality~\eqref{eq:30} is, in particular, satisfied by solutions $u$
of~\eqref{eq:10} (cf.,~\cite{MR0348562,MR3465809}). It is worth
mentioning that inequality~\eqref{eq:30} plays a crucial role in
abstract $2^{\textrm{nd}}$-order problems of elliptic type involving
accretive operators $A$ (see, for example,~\cite[(2.22) on page
525]{MR826651} or, more recently, \cite[(1.8) on page
719]{MR4026441}).

In many applications, the Banach space $X$ is given by the classical
Lebesgue space $(L^{q}:=L^{q}(\Sigma,\mu),\norm{\cdot}_{q})$,
($1\le q\le \infty$), for a given $\sigma$-finite measure space
$(\Sigma,\mu)$. If, in addition, the mild solutions $u$ of Cauchy
problem~\eqref{eq:10bis-no-F} satisfy a global
\emph{$L^{q}$-$L^{r}$ regularity estimate} ($1\le q$, $r\le \infty$,
cf.,~\cite{CoulHau2016})
\begin{equation}
  \label{eq:44}
  \norm{u(t)}_{r}\le
  C\,e^{\omega t}\frac{\norm{u(0)}_{q}^{\gamma}}{t^{\delta}}\qquad
  \text{for all $t>0$,}
\end{equation}
holding for some $C>0$,
$\gamma$, $\delta>0$, then by
combining~\eqref{eq:29} with~\eqref{eq:44} leads to
 \begin{equation}
    \label{eq:16}
    \limsup_{h\to 0+}\frac{\norm{u(t+h)-u(t)}_{r}}{h}
    \le C\,2^{\delta+2}\,e^{\omega\,t}
    \frac{\norm{u_{0}}_{q}^{\gamma}}{t^{\delta+1}}.
\end{equation}
We outline this result in full details in
Corollary~\ref{cor:1}. Regularity estimates similar to~\eqref{eq:44}
have been studied recently by many authors (see, for
example,~\cite{MR1103113,MR1218884,MR2569498} and the references
therein for the linear theory, and we refer to~\cite{CoulHau2016} and
the references therein for the nonlinear one). The idea to combine an
$L^{q}$-$L^{r}$ regularity estimate~\eqref{eq:44} for $q=1$ and $r=\infty$
with the estimate~\eqref{eq:29} was already used by Alikakos and
Rostamian~\cite{MR656651} to obtain \emph{gradient decay estimates}
for solutions of the parabolic $p$-Laplace equation on the
Euclidean space $\R^{d}$. Thus, Corollary~\ref{cor:1} improves this
result to a more general abstract framework with a Lipschitz
perturbation. For further applications, we refer the interested reader
to the book~\cite{CoulHau2016}.

The structure of this paper is as follows. In the subsequent section,
we collect some intermediate results to prove our main theorems
(Theorem~\ref{thm:main1} and Theorem~\ref{thm:2}).

In Section~\ref{sec:semigroups}, we consider the class of \emph{quasi
  accretive operators} $A$ (see Definition~\ref{def:quasi-accretive})
and outline how the property that $A$ is homogeneous of order
$\alpha\neq 1$ is passed on the \emph{nonlinear semigroup}
$\{T_{t}\}_{t\ge 0}$ generated by $-A$ (see the paragraph after
Definition~\ref{def:mild-solution}). In particular, we discuss when
solutions $u$ of~\eqref{eq:10bisF} are differentiable with values in
$X$ at a.e. $t>0$, and give the proofs of Theorem~\ref{thm:main1},
Corollary~\ref{cor:RN-Lipschitz-case}, and Theorem~\ref{thm:2}.

Section~\ref{sec:completely-accretive} focuses on the class of
semigroups generated by a homogenous \emph{quasi completely accretive}
operators $A$ of order $\alpha\neq 1$. The notion of
\emph{completely accretive} operators $A$ (see
 Definition~\ref{def:completely-accretive-operators}) was introduced by
 B\'enilan and Crandall~\cite{MR1164641} and further extended by
 Jakubowski and Wittbold~\cite{MR1980979} to study nonlinear Volterra
 equations governed by this class of operators. More recently,
 Coulhon and the author~\cite{CoulHau2016} introduced the class of
 \emph{quasi completely accretive} operators to study additional regularity
 properties of mild solutions to Cauchy problem~\eqref{eq:10bis-no-F}
 (respectively,~\eqref{eq:10}) when the infinitesimal generator
 satisfies a functional inequality of Sobolev, Gagliardo-Nirenberg, or
 Nash type. We prove in
 Section~\ref{subsec:global-regularity-estimates} a compactness
 result (see Lemma~\ref{lem:compacteness-of-complete-semigroup}) and
 due to this, we obtain in
Theorem~\ref{thm:regularity-of-complete-acc-homogen-operators} that
every mild solution $u$ of the homogeneous Cauchy problem
~\eqref{eq:10} governed by a homogenous quasi completely accretive
operators $A$ of order $\alpha\neq 1$ defined on  also-called
\emph{normal} Banach space, is differentiable for a.e. $t>0$ and
its right-hand side time-derivative satisfies point-wise
  Aronson-B\'enilan type estimates and global $L^1$ Aronson-B\'enilan type estimates.

We conclude this paper in Section~\ref{sec:application} with an
application; we derive in Theorem~\ref{thm:reg-DtN} global $L^{q}$-$L^{\infty}$-regularity
estimates of the time-derivative $\frac{\td u}{\dt}$ for solutions $u$
to the perturbed evolution problem~\eqref{eq:10bisF} when $A$ is the
Dirichlet-to-Neumann operator associated with the negative
$p$-Laplacian $-\Delta_{p}$ plus lower order terms on a compact, smooth, Riemannian
manifold $(M,g)$ with a Lipschitz continuous boundary.

%
%

\section{Preliminaries}
\label{sec:main-results}

In this section, we gather some intermediate results to prove the
main theorems of this paper.\medskip

Suppose $X$ is a linear vector space and $\norm{\cdot}_{X}$ a
semi-norm on $X$. Then, the main object of this paper is the following
class of operators (cf.,~\cite{MR648452} and~\cite{MR4031770}).

\begin{definition}\label{def:1}
  An operator $A$ on $X$ is called \emph{homogeneous of order
    $\alpha\in \R$} if $0\in A0$, and for every $u\in D(A)$ and
  $\lambda\ge 0$, one has that $\lambda u\in D(A)$ and $A$
  satisfies~\eqref{eq:41}.
  \end{definition}

  For the rest of this section suppose that $A$ denotes a homogeneous
  operator on $X$ of order $\alpha\neq 1$. We begin by considering the
  inhomogeneous Cauchy problem
  \begin{equation}
    \label{eq:10bis}
    \begin{cases}
      \frac{\td u}{\dt}+A(u(t))\ni f(t) & \text{for
        a.e. $t\in (0,T)$,}\\
      \;\phantom{\frac{\td u}{\dt}+A(}u(0)=u_{0}, &
    \end{cases}
  \end{equation} 
  and want to discuss the impact of the homogeneity of $A$ on the
  solutions $u$ to~\eqref{eq:10bis}. For this, suppose
  $f\in C([0,T];X)$, $u_{0}\in X$, and $u\in C^{1}([0,T];X)$ be a
  classical solution of~\eqref{eq:10bis}. Further, for given
  $\lambda>0$, set
\begin{displaymath}
  v_{\lambda}(t)=\lambda^{\frac{1}{\alpha-1}} u(\lambda t),\qquad( t\in \left[0,\tfrac{T}{\lambda}\right]).
\end{displaymath}
Then, $v$ satisfies
\begin{align*}
  \frac{\td v_{\lambda}}{\dt}(t)&=\lambda^{\frac{1}{\alpha-1}+1}\frac{\td u}{\dt}(\lambda t)\in
  \lambda^{\frac{\alpha}{\alpha-1}}\Big[f(\lambda t)- A(u(\lambda
    t))\Big] \\
  &=-A(v_{\lambda}(t)) + \lambda^{\frac{\alpha}{\alpha-1}} f(\lambda t)
\end{align*}
for every $t\in (0,T/\lambda)$ with initial value
$v_{\lambda}(0)=\lambda^{\frac{1}{\alpha-1}}
u(0)=\lambda^{\frac{1}{\alpha-1}}u_{0}$. 

Now, if we assume that the Cauchy problem~\eqref{eq:10bis} is
well-posed for given $u_{0}\in \overline{D(A)}^{\mbox{}_{X}}$ and
$f\in L^{1}(0,T;X)$ in the sense that there is a semigroup
$\{T_{t}\}_{t=0}^{T}$ of mappings
$T_{t} : \overline{D(A)}^{\mbox{}_{X}}\times L^{1}(0,T;X)\to
\overline{D(A)}^{\mbox{}_{X}}$ given by
\begin{equation}
  \label{eq:36}
  T_{t}(u_{0},f):=u(t)\qquad\text{for every $u_{0}\in \overline{D(A)}^{\mbox{}_{X}}$ and $f\in L^{1}(0,T;X)$,}
\end{equation}
where $u$ is the unique (mild) solution. Then, the
previous reasoning can be formulated in terms of this semigroup
$\{T_{t}\}_{t=0}^{T}$ as follows
\begin{equation}
  \label{eq:4}
  T_{t}(0,0)=0\qquad\text{for all $t\in [0,T]$}
\end{equation}
(i.e., $u(t)\equiv 0$ is the unique solution of~\eqref{eq:10bis} if
$u_{0}=0$ and $f(t)\equiv 0$), and
\begin{equation}
 \label{eq:21bis}
  \lambda^{\frac{1}{\alpha-1}}T_{\lambda
    t}(u_{0},f)=T_{t}(\lambda^{\frac{1}{\alpha-1}}u_{0},
  \lambda^{\frac{\alpha}{\alpha-1}}f(\lambda\cdot))\quad\text{for
    every $t\in [0,T/\lambda]$, $\lambda>0$.}
\end{equation}
Property~\eqref{eq:21bis} together with the standard growth estimate
\begin{equation}
  \label{eq:20bis}
  \begin{split}
    &e^{-\omega t}\norm{T_{t}(u_{0},f)-T_{t}(\hat{u}_{0},\hat{f})}_{X}\\
    &\qquad\le L e^{-\omega
      s}\norm{T_{s}(u_{0},f)-T_{s}(\hat{u}_{0},\hat{f})}_{X}
    +L \int_{s}^{t}e^{-\omega
      r}\,\norm{f(r)-\hat{f}(r)}_{X}\,\dr\\
    &\hspace{2cm}\text{for every $0 \le s \le t \le T$, $u_{0}\in
      \overline{D(A)}^{\mbox{}_{X}}$, $f$, $\hat{f}\in L^{1}(0,T;X)$,}
  \end{split}
\end{equation}
holding for some $\omega\in \R$ and $L\ge 1$, are the main ingredients
to obtain global regularity estimates of the form~\eqref{eq:29}. This
leads to our first intermediate result. This lemma also generalizes
the case of homogeneous operators of order zero
(cf.,~\cite[Theorem~2.3]{MR4031770}), and the case $\omega=0$ treated
in~\cite[Theorem~4]{MR648452}.

\begin{lemma}\label{thm:1bis}
  Let $\{T_{t}\}_{t=0}^{T}$ be a family of mappings
  $T_{t} : C\times L^{1}(0,T;X)\to C$ defined on a subset
  $C\subseteq X$, and suppose there are $\omega\in \R$, $L\ge 1$, and
  $\alpha\neq 1$ such that $\{T_{t}\}_{t=0}^{T}$
  satisfies~\eqref{eq:4}-\eqref{eq:20bis}. Then, the following
  statements hold. 
  \begin{enumerate}
  \item For every $u_{0}\in C$, $f\in L^{1}(0,T;X)$, $t\in (0,T]$ and
    $h>0$ such that $t+h\in (0,T]$, one has that
    \begin{equation}
      \label{eq:25bis}
      \begin{split}
        &\norm{T_{t+h}(u_{0},f)-T_{t}(u_{0},f)}_{X}\\
        &\quad \le
        \labs{\left(1+\tfrac{h}{t}\right)-\left(1+\tfrac{h}{t}\right)^{\frac{1}{1-\alpha}}}\;L\,\int_{0}^{t}e^{\omega
          (t-s)}\norm{f(s+\tfrac{h}{t}s)}_{X}\,\ds\\
        &\hspace{2cm}
        +\left(1+\tfrac{h}{t}\right)^{\frac{1}{1-\alpha}}\;L\,\int_{0}^{t}e^{\omega
          (t-s)}\norm{f(s+\tfrac{h}{t}s)-f(s)}_{X}\,\ds\\
        &\hspace{2cm} + \;L\,e^{\omega\,t}
        \labs{\left(1+\tfrac{h}{t}\right)^{\frac{1}{1-\alpha}}-1}\left(2\,\norm{u_{0}}_{X}+\int_{0}^{t}e^{-\omega
            s}\norm{f(s)}_{X}\,\ds\right).
      \end{split}
    \end{equation}
    \item If one denotes 
    \begin{equation}
      \label{eq:64}
      V_{\omega}(f,t):=\limsup_{h\to 0+}
      \int_{0}^{t}e^{-\omega\,s}\frac{\norm{f(s+h s)-f(s)}_{X}}{h}\,\ds,
    \end{equation}
    and $\{T_{t}\}_{t=0}^{T}$ satisfies~\eqref{eq:20bis}, then for
    every $t>0$ and $u_{0}\in C$, one has that
    \begin{equation}
      \label{eq:38}
      \begin{split}
        &\limsup_{h\to 0+}\lnorm{\frac{T_{t+h}(u_{0},f)-T_{t}(u_{0},f)}{h}}_{X}\\
        & \qquad \le \frac{L}{t}\,e^{\omega t}
        \left[2\frac{\norm{u_{0}}_{X}}{\abs{1-\alpha}}\,
          +\frac{1}{\abs{1-\alpha}}\int_{0}^{t}e^{-\omega
            s}\norm{f(s)}_{X}\,\ds+V_{\omega}(f,t)\right],
      \end{split}
    \end{equation}
    and if $f\in W^{1,1}(0,T;X)$, then
    \begin{equation}
      \label{eq:39}
      \begin{split}
        &\limsup_{h\to 0+}\lnorm{\frac{T_{t+h}(u_{0},f)-T_{t}(u_{0},f)}{h}}_{X}\\
        & \qquad \le \frac{L}{t}\,e^{\omega t}
        \left[2\frac{\norm{u_{0}}_{X}}{\abs{1-\alpha}}\,
          +\frac{1}{\abs{1-\alpha}}\int_{0}^{t}e^{-\omega
            s}\norm{f(s)}_{X}\,\ds\right.\\
        &\left.\hspace{6cm}+\int_{0}^{t}e^{-\omega
            s}\norm{f'(s)}_{X}\,s\,\ds \right].
      \end{split}
    \end{equation}
   
   \item If for given $u_{0}\in C$ and $f\in W^{1,1}(0,T;X)$,
    $\frac{\td }{\dt}_{\!\! +}\! T_{t}(u_{0},f)$ exists (in $X$) at
    a.e. $t\in (0,T)$, then
    \begin{equation}
      \label{eq:40}
      \begin{split}
        \lnorm{\frac{\td }{\dt}_{\!\! +}\! T_{t}(u_{0},f)}_{X} & \le
        \frac{L}{t}\,e^{\omega t}
        \left[2\frac{\norm{u_{0}}_{X}}{\abs{1-\alpha}}\,
          +\frac{1}{\abs{1-\alpha}}\int_{0}^{t}e^{-\omega
            s}\norm{f(s)}_{X}\,\ds\right.\\
        &\left.\hspace{4.4cm}+\int_{0}^{t}e^{-\omega
            s}\norm{f'(s)}_{X}\,s\,\ds \right].
      \end{split}
    \end{equation}
  \end{enumerate}
\end{lemma}

Our proof of Lemma~\ref{thm:1bis} uses the same techniques as
in~\cite{MR648452}.
\mbox{}\medskip

\allowdisplaybreaks
\begin{proof}
  Let $u_{0}\in C$, $f\in L^{1}(0,T;X)$, $t>0$, and $h>0$
  satisfying $t+h\le T$. If we choose
  $\lambda=1+\frac{h}{t}$ in~\eqref{eq:21bis}, then
  \begin{equation}
    \label{eq:16bis}
    \begin{split}
      &T_{t+h}(u_{0},f)-T_{t}(u_{0},f)\\
      &\qquad =T_{\lambda
        t}(u_{0},f)-T_{t}(u_{0},f)\\
      &\qquad= \lambda^{\frac{1}{1-\alpha}}
      T_{t}\left[\lambda^{\frac{1}{\alpha-1}} u_{0},\lambda^{\frac{\alpha}{\alpha-1}}
      f(\lambda\cdot)\right]-T_{t}(u_{0},f)
    \end{split}
  \end{equation}
  and so,
  \begin{equation}
  \begin{split}\label{eq:42}
    &T_{t+h}(u_{0},f)-T_{t}(u_{0},f)\\
    & \qquad =
        \lambda^{\frac{1}{1-\alpha}}
        \left[T_{t}\left[\lambda^{\frac{1}{\alpha-1}} u_{0},\lambda^{\frac{\alpha}{\alpha-1}}
      f(\lambda\cdot)\right]-T_{t}(u_{0}, f(\lambda\cdot)\right]\\
                     &\hspace{2cm} + \lambda^{\frac{1}{1-\alpha}}\,
                       \left[T_{t}\left[u_{0},f(\lambda\cdot)\right]-T_{t}(u_{0},f)\right]\\
                     &\hspace{4cm}    + \left[\lambda^{\frac{1}{1-\alpha}}-1\right]\, T_{t}(u_{0},f).
  \end{split}
\end{equation}
  Applying to this~\eqref{eq:20bis} and by using~\eqref{eq:4},
  one sees that
\begin{align*}
  &\norm{T_{t+h}(u_{0},f)-T_{t}(u_{0},f)}_{X}\\
  &\qquad\le
  \left(1+\tfrac{h}{t}\right)^{\frac{1}{1-\alpha}}
                                      \,\lnorm{T_{t}\left[\lambda^{\frac{1}{\alpha-1}} u_{0},\lambda^{\frac{\alpha}{\alpha-1}}
      f(\lambda\cdot)\right]-T_{t}(u_{0}, f(\lambda\cdot)}_{X}\\
  &\hspace{2cm} + \left(1+\tfrac{h}{t}\right)^{\frac{1}{1-\alpha}}
    \lnorm{T_{t}\left[u_{0},f(\lambda\cdot)\right]-T_{t}(u_{0},f)}_{X}\\
  &\hspace{4cm}    +
    \labs{\left(1+\tfrac{h}{t}\right)^{\frac{1}{1-\alpha}}-1}\,\norm{T_{t}(u_{0},f)}_{X}\\
  &\qquad\le\left(1+\tfrac{h}{t}\right)^{\frac{1}{1-\alpha}}\;L\,e^{\omega\,t}
    \lnorm{\left(1+\tfrac{h}{t}\right)^{\frac{1}{\alpha-1}}u_{0}-u_{0}}_{X}\\
  &\hspace{1cm} +
    \left(1+\tfrac{h}{t}\right)^{\frac{1}{1-\alpha}}\;L\,\int_{0}^{t}e^{\omega
    (t-s)}\lnorm{(1+\tfrac{h}{t})^{\frac{\alpha}{\alpha-1}}f(s+\tfrac{h}{t}s)-f(s+\tfrac{h}{t}s)}_{X}\,\ds\\
 &\hspace{1cm} +\left(1+\tfrac{h}{t}\right)^{\frac{1}{1-\alpha}}\;L\,\int_{0}^{t}e^{\omega
    (t-s)}\norm{f(s+\tfrac{h}{t}s)-f(s)}_{X}\,\ds\\
  &\hspace{2cm}    + \;L\,e^{\omega\,t}
    \labs{\left(1+\tfrac{h}{t}\right)^{\frac{1}{1-\alpha}}-1}\left(\norm{u_{0}}_{X}+\int_{0}^{t}e^{-\omega
    s}\norm{f(s)}_{X}\,\ds\right)\\
   &\qquad=
      \labs{\left(1+\tfrac{h}{t}\right)-\left(1+\tfrac{h}{t}\right)^{\frac{1}{1-\alpha}}}\;L\,\int_{0}^{t}e^{\omega
      (t-s)}\norm{f(s+\tfrac{h}{t}s)}_{X}\,\ds\\
  &\hspace{2cm} +\left(1+\tfrac{h}{t}\right)^{\frac{1}{1-\alpha}}\;L\,\int_{0}^{t}e^{\omega
    (t-s)}\norm{f(s+\tfrac{h}{t}s)-f(s)}_{X}\,\ds\\
  &\hspace{3cm}    + \;L\,e^{\omega\,t}
    \labs{\left(1+\tfrac{h}{t}\right)^{\frac{1}{1-\alpha}}-1}\left(2\,\norm{u_{0}}_{X}+\int_{0}^{t}e^{-\omega
    s}\norm{f(s)}_{X}\,\ds\right),
\end{align*}
which is~\eqref{eq:25bis}. It is clear
that~\eqref{eq:38}-\eqref{eq:40} follow from \eqref{eq:25bis}.
\end{proof}

Examples of functions $f : [0,T]\to X$ for which $V_{\omega}(f,t)$
defined by~\eqref{eq:64} is finite at a.e. $t$ and integrable on $L^{1}(0,T)$, are functions with
\emph{bounded variation} (cf.,~\cite[Appendice, Section 2.]{MR0348562}).

\begin{definition}\label{def:BV}
  For a function $f : [0,T]\to X$, one calls 
  \begin{displaymath}
    \Var(f;[0,T]):=\sup\Big\{\sum_{i=1}^{N}\norm{f(t_{i})-f(t_{i-1})}_{X}\,\Big\vert\,
    \begin{array}[c]{c}
\text{all partitions : }\\
      0=t_{0}<\cdots<t_{N}=T
    \end{array}
    \Big\}
  \end{displaymath}
  the \emph{total variation of $f$}. Each $X$-valued function $f : [0,T]\to X$ is
  said to have  \emph{bounded variation} on $[0,T]$ if $\Var(f;[0,T])$
  is finite. We denote by $BV(0,T;X)$ the space of all functions $f :
  [0,T]\to X$ of bounded variation and to simplify the notation, we
  set $V_{f}(t)=\Var(f;[0,t])$ for $t\in (0,T]$.
\end{definition}

Functions of bounded variation have the following properties.

\begin{proposition}\label{prop:BVX}
  Let $f\in BV(0,T;X)$. Then the following statements hold.
  \begin{enumerate}
     \item $f\in L^{\infty}(0,T;X)$;
     \item At every $t\in [0,T]$, the left-hand side limit
       $f(t-):=\lim_{s\to t-}f(s)$ and right-hand side limit
       $f(t+):=\lim_{s\to t+}f(s)$ exist in $X$; and the set of
       discontinuity points in $[0,T]$ is at most countable;
     \item The mapping $t\mapsto V_{f}(t)$ is monotonically increasing
       on $[0,T]$, and
       \begin{equation}
         \label{eq:58}
         \norm{f(t)-f(s)}_{X}\le V_{f}(t)-V_{f}(s)\qquad\text{for all
           $0\le s\le t\le T$;}
       \end{equation}
     \item \label{prop:BVX-4} For $\omega\ge 0$, one has that
       \begin{displaymath}
         \label{eq:60}
         \int_{0}^{t}e^{-\omega s}\frac{\norm{f(s+hs)-f(s)}_{X}}{h}\ds\le
         t\,V_{f}(t)\qquad\text{for all $h\in (0,t]$, $0<t\le T$.}
       \end{displaymath}
     \item\label{prop:BVX-5} For $\omega\ge 0$, let $V_{\omega}(f,t)$
       be given by~\eqref{eq:64}. Then $V_{\omega}(f,t)$ belongs to
       $L^{\infty}(0,T)$ satisfying
       \begin{displaymath}
         V_{\omega}(f,t)\le t\,\, V_{f}(t)\qquad\text{for all
           $t\in [0,T]$}.
       \end{displaymath}
  \end{enumerate}
\end{proposition}


The first three statements are standard and can be found, for
  example, in~\cite[Section 2., Lemme A.1]{MR0348562}. Thus, we
  only outline the proof of statement~\eqref{prop:BVX-4} and \eqref{prop:BVX-5}.

\begin{proof}
  Obviously, \eqref{prop:BVX-5} follows from~\eqref{prop:BVX-4}. Thus,
  it remains to show that for given $f\in BV(0,T;X)$,
  \eqref{prop:BVX-4} holds. To see this, let
  $t\in (0,T)$, $h\in (0,t]$ such that $t+h\le T$.
  Then, by~\eqref{eq:58} and since $\omega\ge 0$,
  \begin{align*}
    \int_{0}^{t}e^{-\omega s}\frac{\norm{f(s+h
      s)-f(s)}_{X}}{h} \ds
     &\le \tfrac{1}{h}\int_{0}^{t}e^{-\omega
       s}\Big(V_{f}((1+h)s)-V_{f}(s)\Big)\ds\\
    &\le \tfrac{1}{h}\int_{0}^{t}\Big(V_{f}((1+h)s)-V_{f}(s)\Big)\ds.
\end{align*}
 By using the substitution $r=(1+h)s$, we get
\begin{align*}
  &\tfrac{1}{h}\int_{0}^{t} V_{f}((1+h)s)\,\ds -
    \tfrac{1}{h}\int_{0}^{t} V_{f}(s)\,\ds\\
  &\qquad = \tfrac{1}{h(1+h)}\int_{0}^{(1+h)t} V_{f}(r)\,\dr -
    \tfrac{1}{h}\int_{0}^{t} V_{f}(s)\,\ds\\
  &\qquad \le \tfrac{1}{h}\int_{t}^{t+ht} V_{f}(s)\,\ds
\end{align*}
and by the monotonicity of $t\mapsto V_{f}(t)$,
\begin{displaymath}
  \tfrac{1}{h}\int_{t}^{t+ht} V_{f}(s)\,\ds\le t\, V_{f}(t).
\end{displaymath}
This shows that~\eqref{eq:60} holds.
\end{proof}

In the case $f\equiv 0$, we let $T=\infty$. Then the mapping $T_{t}$ given
by~\eqref{eq:36} only depends on the initial value $u_{0}$. In other words,
\begin{equation}
  \label{eq:55}
  T_{t}u_{0}=T_{t}(u_{0},0)\qquad\text{for every $u_{0}\in C$ and
    $t\ge 0$.}
\end{equation}
In this case Lemma~\ref{thm:1bis} reads as follows (cf.,~\cite{MR1164641}).

\begin{corollary}\label{thm:1}
  Let $\{T_{t}\}_{t\ge 0}$ be a family of mappings $T_{t} : C\to C$
  defined on a subset $C\subseteq X$, and suppose there are $\omega\in
  \R$, $L\ge 1$, and
  $\alpha\neq 1$ such that $\{T_{t}\}_{t\ge 0}$ satisfies
 \begin{align}
  \label{eq:20}
  \norm{T_{t}u_{0}-T_{t}\hat{u}_{0}}_{X}&\le L\,e^{\omega
    t}\,\norm{u _{0}-\hat{u}_{0}}_{X}\qquad\text{for all $t\ge 0$,
                                          $u$, $\hat{u}\in C$,}\\
   \label{eq:21}
  \lambda^{\frac{1}{\alpha-1}}\,T_{\lambda t}u _{0} &=T_{t}[\lambda^{\frac{1}{\alpha-1}}u _{0}]\qquad\text{for all
    $\lambda>0$, $t\ge 0$
    and $u _{0}\in C$.}
 \end{align}
 Further, suppose
 $T_{t}0\equiv 0$ for all $t\ge 0$. Then, for every $u_{0}\in C$,
 \begin{equation}
   \label{eq:25}
   \norm{T_{t+h}u_{0}-T_{t}u_{0}}_{X}\le
   2\,L\,\labs{1-\left(1+\tfrac{h}{t}\right)^{\frac{1}{1-\alpha}}}\,
   e^{\omega\,t}\norm{u_{0}}_{X}.
 \end{equation}
 $t>0$, $h\neq 0$ satisfying $1+\frac{h}{t}>0$. In
 particular, the family $\{T_{t}\}_{t\ge 0}$ satisfies
 \begin{equation}
   \label{eq:18bis}
   \limsup_{h\to 0+}\frac{\norm{T_{t+h}u_{0}-T_{t}u_{0}}_{X}}{h}\le
   \frac{2 L e^{\omega t}}{\abs{1-\alpha}}\frac{\norm{u_{0}}_{X}}{t}\qquad\text{for every $t>0$, $u_{0}\in C$.}
 \end{equation}
 Moreover, if  for $u_{0}\in C$, the right-hand side derivative
 $\frac{\td T_{t}u_{0}}{\dt}_{+}$ exists (in $X$) at $t>0$, then
 \begin{equation}
   \label{eq:18bis2}
   \lnorm{\frac{\td T_{t}u_{0}}{\dt_{\!+}}}_{X}\le
   \frac{2\,L\,e^{\omega t}}{\abs{1-\alpha}}\frac{\norm{u_{0}}_{X}}{t}.
 \end{equation}
\end{corollary}

Finally, we turn to the Cauchy problem  governed by the
operator $A+F$,
\begin{equation}
  \label{eq:18}
  \begin{cases}
    \frac{\td u}{\dt}+A(u(t))+ F(u(t))\ni f(t) &\text{on
      $(0,T)$,}\\
    \;\phantom{\frac{\td u}{\dt}+A(u(t))+ F(}u(0)=u_{0}, &
  \end{cases}
\end{equation}
for given $u_{0}\in \overline{D(A)}^{\mbox{}_{X}}$ and
$f\in L^{1}(0,T;X)$, involving a homogenous operator
$A$ in $X$ of order $\alpha\neq 1$, and a Lipschitz continuous
perturbation $F : X\to X$ with Lipschitz constant $\omega\ge 0$
satisfying $F(0)=0$. We assume that Cauchy problem~\eqref{eq:18} is
well-posed in $X$ in the sense that for every $u_{0}\in
\overline{D(A)}^{\mbox{}_{X}}$ and $f\in L^{1}(0,T;X)$, there is a
unique function $u\in C([0,T];X)$ satisfying $u(0)=u_{0}$ in $X$ and
\eqref{eq:36} generates a semigroup
$\{T_{t}\}_{t=0}^{T}$ of mappings $T_{t} :
\overline{D(A)}^{\mbox{}_{X}}\times L^{1}(0,T;X) \to
\overline{D(A)}^{\mbox{}_{X}}$ satisfying~\eqref{eq:20bis} for every
$0\le s<t\le T$.

One important idea to obtain global $L^1$ Aronson-B\'enilan type estimates for the semigroup $\{T_{t}\}_{t=0}^{T}$
associated with~\eqref{eq:18} is the assumption that for given
$u_{0}\in \overline{D(A)}^{\mbox{}_{X}}$ and $f\in L^{1}(0,T;X)$, the
unique solution $t\mapsto u(t)=T_{t}(u_{0},f)$ of Cauchy problem~\eqref{eq:18} is, in particular, the unique
solution of the unperturbed inhomogeneous Cauchy
problem~\eqref{eq:10bis} for $\tilde{f} : [0,T]\to X$
given by
\begin{equation}
  \label{eq:43}
  \tilde{f}(t):=f(t)-F(T_{t}(u_{0},f)),\qquad\text{($t\in [0,T]$).}
\end{equation}
This property can be expressed by the identity
\begin{equation}
  \label{eq:45}
    \tilde{T}_{t}(u_{0},\tilde{f})=T_{t}(u_{0},f)\qquad\text{holds for
      every $t\in [0,T]$,}
\end{equation}
where $\{\tilde{T}_{t}\}_{t=0}^{T}$ denotes the semigroup associated
with~\eqref{eq:10bis}. The advantage of equation~\eqref{eq:45} is that one can employ
inequality~\eqref{eq:20bis} satisfied by 
the family $\{\tilde{T}_{t}(\cdot,\tilde{f})\}_{t\ge 0}$. Thus, by Lemma~\ref{thm:1bis}, the following
estimate holds. 

\begin{theorem}\label{cor:1bisF}
  Let $F : X\to X$ be a Lipschitz continuous mapping with Lipschitz
  constant $\omega_{F}\ge 0$ satisfying $F(0)=0$. Given $T>0$ and a
  subset $C\subseteq X$, assume there are families
  $\{T_{t}\}_{t=0}^{T}$ and $\{\tilde{T}_{t}\}_{t=0}^{T}$ of mappings
  $T_{t}$, $\tilde{T}_{t} : C\times L^{1}(0,T;X)\to C$
  satisfying~\eqref{eq:4} and related through~\eqref{eq:45} for every
  $u_{0}\in C$ and $f\in L^{1}(0,T;X)$ with $\tilde{f}$ given
  by~\eqref{eq:43}. Further suppose, $\{\tilde{T}_{t}\}_{t=0}^{T}$
  satisfies~\eqref{eq:21bis} and \eqref{eq:20bis} for some
  $\omega\ge 0$ and $L\ge 1$, and $\{T_{t}\}_{t=0}^{T}$
  satisfies~\eqref{eq:20bis} with $\tilde{\omega}=\omega+\omega_{F}$
  and $L$.

  Then, if for $u_{0}\in C$ and $f\in BV(0,T;X)$, the function
  $t\mapsto T_{t}(u_{0},f)$ is locally Lipschitz continuous on $[0,T)$,
  then one has that
 \begin{equation}
    \label{eq:48}
    \limsup_{h\to 0+}\frac{\norm{T_{t+h}(u_{0},f)-T_{t}(u_{0},f)}}{h}
    \le\frac{e^{\omega t}}{t}\!\! \left[ a(t)+ L \omega_{F}\!\!\int_{0}^{t}a(s)
      e^{ L\,\omega_{F}(t-s)}\ds\right]
  \end{equation}
  for a.e. $t\in (0,T)$, where
  \begin{equation}
    \label{eq:34}
    \begin{split}
      a(t)&:= L\,V_{\omega}(f,t) + \frac{L}{\abs{1-\alpha}}\left[\left(2+\omega_{F}\,L\int_{0}^{t}e^{\omega_{F}
    s}\ds \right)\,\norm{u_{0}}_{X}\right.\\
  &\hspace{1cm}\left. +\int_{0}^{t}e^{-\omega
    s}\norm{f(s)}_{X}\,\ds +\omega_{F}\,L\,\int_{0}^{t} \int_{0}^{s}e^{-\omega_{F} r} \norm{f(r)}_{X}\dr\,\ds\right].  
    \end{split}
  \end{equation}
and $V_{\omega}(f,\cdot)$ is given by~\eqref{eq:64}.
\end{theorem}

For the proof of this theorem, we still need the following version of
Gronwall's lemma.

\begin{lemma}[{\cite[Lemma~D.2]{MR1873467}}]
  \label{lem:1}
    Suppose $v \in L^{1}_{loc}([0,T))$ satisfies
    \begin{equation}
      \label{eq:23}
      v(t)\le a(t)+\int_{0}^{t}v(s)\,b(s)\ds\qquad\text{for a.e. $t\in (0,T)$,}
    \end{equation}
    where $b \in C([0,T))$ satisfying $b(t)\ge 0$, and
    $a \in L^{1}_{loc}([0,T))$. Then,
    \begin{equation}
      \label{eq:24}
      v(t)\le a(t)+\int_{0}^{t}a(s)\,b(s)\,e^{\int_{s}^{t}b(r)\dr}\,\ds
      \qquad\text{for a.e. $t\in (0,T)$.}
    \end{equation}
\end{lemma}

We are now ready to give the proof of Theorem~\ref{cor:1bisF}.

\begin{proof}[Proof of Theorem~\ref{cor:1bisF}]
  Let $u_{0}\in C$ and $f\in BV(0,T;X)$. Fix $t>0$, and
  let $h>0$ such that $t+h< T$. Then, by the assumption that there
  is a family $\{\tilde{T}_{t}\}_{t=0}^{T}$ of mappings
  $\tilde{T}_{t}$ satisfying~\eqref{eq:45} for every $u_{0}\in C$ and
  $f\in L^{1}(0,T;X)$ with $\tilde{f}$ given
  by~\eqref{eq:43}, and $\{\tilde{T}_{t}\}_{t=0}^{T}$
  satisfies~\eqref{eq:4}-\eqref{eq:20bis} for some
  $\omega\ge 0$, $L$, we can apply
  Lemma~\ref{thm:1bis} to $\tilde{T}_{t}(u_{0},\tilde{f})$.
  Then by~\eqref{eq:25bis}, since $\tilde{f}$ is given by~\eqref{eq:43}, by~\eqref{eq:45}, 
  and by the triangle inequality,
  \begin{align*}
    &\norm{T_{t+h}(u_{0},f)-T_{t}(u_{0},f)}_{X}\\
    &\; = \norm{\tilde{T}_{t+h}(u_{0},\tilde{f})-\tilde{T}_{t}(u_{0},\tilde{f})}_{X}\\
    &\; \le
    \labs{\left(1+\tfrac{h}{t}\right)-\left(1+\tfrac{h}{t}\right)^{\frac{1}{1-\alpha}}}
      L\!\int_{0}^{t}e^{\omega
    (t-s)}\norm{f(s+\tfrac{h}{t}s)-F(T_{s+\frac{h}{t}s}(u_{0},f))}_{X}\,\ds\\
    &\; +\left(1+\tfrac{h}{t}\right)^{\frac{1}{1-\alpha}}\;L\,\int_{0}^{t}e^{\omega
   (t-s)}\norm{f(s+\tfrac{h}{t}s)-f(s) }_{X}\,\ds\\
   &\;  +\left(1+\tfrac{h}{t}\right)^{\frac{1}{1-\alpha}}\;L\,\int_{0}^{t}e^{\omega
   (t-s)}\norm{F(T_{s+\tfrac{h}{t}s}(u_{0},f))-F(T_{s}(u_{0},f))}_{X}\,\ds\\
  &\;  + \;L\,e^{\omega\,t}
    \labs{\left(1+\tfrac{h}{t}\right)^{\frac{1}{1-\alpha}}-1}\left[2\,\norm{u_{0}}_{X}+\int_{0}^{t}e^{-\omega
    s}\Big[\norm{f(s)}_{X}+\norm{F(T_{s}(u_{0},f))}_{X}\Big]\ds\right]  
  \end{align*}
Since $F$ is globally Lipschitz continuous with constant $\omega_{F}$,
$F(0)=0$, and since $\{T_{t}\}_{t=0}^{T}$ satisfies
\eqref{eq:4} and~\eqref{eq:20bis} with $\tilde{\omega}=\omega+\omega_{F}$ and $L$,
one has that
\begin{displaymath}
  \norm{F(T_{s}(u_{0},f))}_{X}
  \le \omega_{F}\,L\,
    \Big[e^{\tilde{\omega} s}\norm{u_{0}}_{X}+\int_{0}^{s}e^{\tilde{\omega} (s-r)}\norm{f(r)}_{X}\,\dr\Big].
\end{displaymath}
We apply this to the last integral on the right-hand side of the
previous estimate, and substitute $y=(1+h/t)s$ into the first
integral on the right-hand side of the
previous estimate. Then, dividing by $h>0$ both sides
in the resulting inequality yields that
\begin{equation}
  \label{eq:33}
  \begin{split}
     &\frac{\norm{ T_{t+h}(u_{0},f)-T_{t}(u_{0},f)}_{X}}{h}\\
  &\;\le
    \labs{\frac{(1+\tfrac{h}{t})-(1+\frac{h}{t})^{\frac{1}{1-\alpha}}}{\frac{1}{t}h}}
    \frac{1+\frac{h}{t}}{t}
    L e^{\omega t}\times\\
  &\hspace{2cm} \times \int_{0}^{t+h}\! e^{-\frac{\omega}{1+h/t}
    y}\norm{f(y)-F(T_{y}(u_{0},f))}_{X} \dy\\
 &\hspace{0.5cm}
   +\left(1+\tfrac{h}{t}\right)^{\frac{1}{1-\alpha}}\,L\,
   \tfrac{e^{\omega t}}{t}\int_{0}^{t}e^{-\omega
   s}\frac{\norm{f(s+\tfrac{h}{t}s)-f(s)}_{X}}{\frac{h}{t}}\,\ds\\
   &\hspace{0.5cm}
   +\left(1+\tfrac{h}{t}\right)^{\frac{1}{1-\alpha}}\,
   L\, e^{\omega t}\omega_{F}\,\int_{0}^{t}e^{-\omega
   s}\frac{\norm{ T_{s+\tfrac{h}{t}s}(u_{0},f)-T_{s}(u_{0},f)}_{X}}{\frac{s}{t}h}\tfrac{s}{t}\,\ds\\
  &\hspace{0.5cm}   + \tfrac{L e^{\omega\,t}}{t}
    \labs{\frac{(1+\tfrac{h}{t})^{\frac{1}{1-\alpha}}-1}{\frac{1}{t}h}}\left[\left(2+\omega_{F}L
        \int_{0}^{t}e^{\omega_{F} s}\ds \right)\,\norm{u_{0}}_{X}\right.\\
  &\hspace{1cm}\left. +\int_{0}^{t}e^{-\omega
    s}\norm{f(s)}_{X}\,\ds+\omega_{F}\,L\,\int_{0}^{t}
  \int_{0}^{s}e^{-\omega_{F} r}\norm{f(r)}_{X}\dr\,\ds\right],        
\end{split}
\end{equation}
where we use twice that $e^{-\omega s}e^{\tilde{\omega} s}=e^{\omega_{F} s}$. Note that
\begin{displaymath}
  \limsup_{h\to 0+}\int_{0}^{t}e^{-\omega
   s}\frac{\norm{f(s+\tfrac{h}{t}s)-f(s)}_{X}}{\frac{h}{t}}\,\ds=V_{\omega}(f,t)
\end{displaymath}
and by Proposition~\ref{prop:BVX}, one has that
$V_{\omega}(f,\cdot)\in L^{\infty}([0,T))$. Since
$t\mapsto T_{t}(u_{0},f)$ is locally Lipschitz continuous on $[0,T)$,
for every $\varepsilon\in (0,T)$ there is a constant $C_{\varepsilon}>0$ such that
\begin{displaymath}
  \frac{\lnorm{
      T_{s+\tfrac{h}{t}s}(u_{0},f)-T_{s}(u_{0},f)}_{X}}{\frac{s}{t}h}\le C
\end{displaymath}
for every $s\in [0,T-\varepsilon]$ and $h>0$ satisfying
$s+\tfrac{h}{t}s<T-\varepsilon$.  Thus, by the reverse version of
Fatou's lemma, taking in~\eqref{eq:33} the limit-superior as $h\to 0+$
gives
\begin{align*}
&e^{-\omega t} t\,  \limsup_{h\to 0+}\frac{\norm{T_{t+h}(u_{0},f)-T_{t}(u_{0},f)}}{h}\\
  &\hspace{1cm}\le L\,V_{\omega}(f,t)
    + L\, \omega_{F}\,
    \int_{0}^{t}e^{-\omega s}s\,\left[ \limsup_{h\to 0+}\frac{\norm{T_{s+h}(u_{0},f)-T_{s}(u_{0},f)}}{h}\right]
    \,\ds\\
  &\hspace{2.5cm} + \frac{L}{\abs{1-\alpha}}\left[\left(2+\omega_{F}\,L\int_{0}^{t}e^{\omega_{F}
    s}\ds \right)\,\norm{u_{0}}_{X}\right.\\
  &\hspace{3cm}\left. +\int_{0}^{t}e^{-\omega
    s}\norm{f(s)}_{X}\,\ds +\omega_{F}\,L\,\int_{0}^{t} \int_{0}^{s}e^{-\omega_{F} r} \norm{f(r)}_{X}\dr\,\ds\right].        
\end{align*}
Now, applying Gronwall's lemma (Lemma~\ref{lem:1}) to $a(t)$ given by~\eqref{eq:34},
  \begin{align*}
    b(t)&\equiv L\,\omega_{F},\text{ and}\\
    v(t)&=e^{-\omega t} t\, \limsup_{h\to 0+}\frac{\norm{T_{t+h}(u_{0},f)-T_{t}(u_{0},f)}}{h},
  \end{align*}
  then one obtains
~\eqref{eq:48}. This completes the proof.
\end{proof}

Next, we intend to extrapolate the regularity estimate~\eqref{eq:48}
for $f\equiv 0$.

\begin{corollary}\label{cor:1}
  Let $\{T_{t}\}_{t\ge 0}$ be a semigroup of mappings $T_{t} : C\to C$
  defined on a subset $C\subseteq X$ and suppose, there is a second
  vector space $Y$ with semi-norm $\norm{\cdot}_{Y}$ and constants $M$, $\gamma$, $\delta>0$ and
  $\hat{\omega}\in \R$ such that
  $\{T_{t}\}_{t\ge 0}$ satisfies the following $Y$-$X$-regularity estimate
  \begin{equation}\label{eq:9}
    \norm{T_{t}u_{0}}_{X}\le
    M\,e^{\hat{\omega} t}\frac{\norm{u_{0}}_{Y}^{\gamma}}{t^{\delta}}\qquad\text{for
      every $t>0$ and $u_{0}\in C\cap Y$.}
  \end{equation}
  If for $\alpha\neq 1$, $\omega$, $\omega_{F}\in \R$ and $L\ge 1$, $\{T_{t}\}_{t\ge 0}$
  satisfies
  \begin{equation}
    \label{eq:48extra}
    \begin{split}
      &\limsup_{h\to 0+}\frac{\norm{T_{t+h}u_{0}-T_{t}u_{0}}_{X}}{h}\\
      &\qquad\qquad \le\frac{e^{\omega t}}{t}\frac{L}{\abs{1-\alpha}}\!\! \left[
        b(t)\,+ L \omega_{F}\!\!\int_{0}^{t} b(s)\,e^{
          L\,\omega_{F}(t-s)}\ds\right]\,\norm{u_{0}}_{X}
    \end{split}
  \end{equation}
  for a.e. $t>0$ and $u_{0}\in C$, with $ b(t):=2+\omega_{F}\,L\int_{0}^{t}e^{\omega_{F} s}\ds$,
  then
  \begin{equation}
    \label{eq:50}
    \begin{split}
      &\limsup_{h\to 0+}\frac{\norm{T_{t+h}u_{0}-T_{t}u_{0}}_{X}}{h}\\
      &\qquad\qquad \le \frac{2^{\delta+1}\,e^{\frac{\omega+\hat{\omega}}{2} t}}{t^{\delta+1}}\frac{L \,M}{\abs{1-\alpha}}\!\! \left[
        b(\tfrac{t}{2})\,+ L \omega_{F}\!\!\int_{0}^{\frac{t}{2}} b(s)\,e^{
          L\,\omega_{F}(\frac{t}{2}-s)}\ds\right]\, \norm{u_{0}}_{Y}^{\gamma}.
    \end{split}
  \end{equation}
  In particular, if the right-hand side derivative
 $\frac{\td }{\dt_{\!+}}T_{t}u_{0}$ exists (in $X$) at $t>0$, then
  \begin{displaymath}
    \begin{split}
      \lnorm{\frac{\td T_{t}u_{0}}{\dt _{+}}}_{X}&\le
      \frac{2^{\delta+1}\,e^{\frac{\omega+\hat{\omega}}{2} t}}{t^{\delta+1}}\frac{L \,M}{\abs{1-\alpha}}\!\! \left[
        b(\tfrac{t}{2})\,+ L \omega_{F}\!\!\int_{0}^{\frac{t}{2}} b(s)\,e^{
          L\,\omega_{F}(\frac{t}{2}-s)}\ds\right]\, \norm{u_{0}}_{Y}^{\gamma}.
    \end{split}
  \end{displaymath}
\end{corollary}

\begin{proof}
  Let $u_{0}\in C$ and $t>0$. Note, if $u_{0}\notin Y$ then
  \eqref{eq:50} trivially holds. Thus, it is sufficient to consider
  the case $u_{0}\in C\cap Y$. By the semigroup property of
  $\{T_{t}\}_{t\ge 0}$ and by~\eqref{eq:48extra} and~\eqref{eq:9}, one
  sees that
  \begin{align*}
    &\limsup_{h\to 0+}\frac{\norm{T_{t+h}u_{0}-T_{t}u_{0}}_{X}}{h}\\
    &\qquad = \limsup_{h\to
      0+}\frac{\norm{T_{\frac{t}{2}+h}(T_{\frac{t}{2}}u_{0})-T_{\frac{t}{2}}(T_{\frac{t}{2}}u_{0})}_{X}}{h}\\
    &\qquad\le\frac{2\,e^{\omega \frac{t}{2}}}{t}\frac{L}{\abs{1-\alpha}}\!\! \left[
        b(\tfrac{t}{2})\,+ L \omega_{F}\!\!\int_{0}^{\frac{t}{2}} b(s)\,e^{
          L\,\omega_{F}(\frac{t}{2}-s)}\ds\right]\,\norm{T_{\frac{t}{2}}u_{0}}_{X}\\
    &\qquad\le \frac{2^{\delta+1}\,e^{\frac{\omega+\hat{\omega}}{2} t}}{t^{\delta+1}}\frac{L \,M}{\abs{1-\alpha}}\!\! \left[
        b(\tfrac{t}{2})\,+ L \omega_{F}\!\!\int_{0}^{\frac{t}{2}} b(s)\,e^{
          L\,\omega_{F}(\frac{t}{2}-s)}\ds\right]\, \norm{u_{0}}_{Y}^{\gamma}.
  \end{align*}
\end{proof}

Now, we suppose, there is a partial ordering ``$\le$'' on $X$ such that
$(X,\le)$ is an ordered vector space. Then, we can state the following
theorem.

\begin{theorem}
  \label{thm:2bis}
  Let $(X,\le)$ be an ordered vector
  space and $F : X\to X$ a Lipschitz continuous mapping satisfying $F(0)=0$. Suppose,
  there is a subset $C\subseteq X$ and two families $\{T_{t}\}_{t\ge
    0}$ and $\{\tilde{T}_{t}\}_{t\ge
    0}$ of mappings $T_{t} : C\to C$ and $\tilde{T}_{t} : C\times
  L^{1}_{loc}([0,\infty);X)\to C$ related by the equation
  \begin{equation}
    \label{eq:47}
    T_{t}u_{0}=\tilde{T}_{t}(u_{0},\tilde{f})\qquad\text{for all $t\ge
      0$, $u_{0}\in C$,}
  \end{equation}
  where $\tilde{f}$ is given by
  $\tilde{f}(t)=-F(T_{t}u_{0})$. Further, suppose 
  \begin{equation}\label{eq:23}
    \text{for every $u_{0}$, $\hat{u}_{0}\in C$ satisfying
      $u_{0}\le\hat{u}_{0}$, one has }T_{t}u_{0}\le
    T_{t}\hat{u}_{0}\text{ for all $t\ge 0$}
  \end{equation}
  and $\{\tilde{T}_{t}\}_{t\ge 0}$ satisfies
  \eqref{eq:4}-\eqref{eq:20bis} for some $\omega\ge 0$ and $L\ge 1$.
  Then for every $u_{0}\in C$ satisfying $u_{0}\ge 0$, one has that
 \begin{equation}
   \label{eq:24}
   \frac{T_{t+h}u_{0}-T_{t}u_{0}}{h}\ge
   \frac{(1+\frac{h}{t})^{\frac{1}{1-\alpha}}-1}{h}\frac{T_{t}u_{0}}{t}+g_{h}(t)
 \end{equation}
 for every $t$, $h>0$ if $\alpha>1$ and
 \begin{equation}
   \label{eq:13}
   \frac{T_{t+h}u_{0}-T_{t}u_{0}}{h}\le
   \frac{(1+\frac{h}{t})^{\frac{1}{1-\alpha}}-1}{h}\frac{T_{t}u_{0}}{t}+g_{h}(t)
 \end{equation}
 for every $t$, $h>0$ if $\alpha<1$, where for every $h>0$,
 $g_{h} : (0,\infty)\to X$ is a continuous function satisfies
 \begin{equation}
   \label{eq:15}
   \begin{split}
     &\norm{g_{h}(t)}_{X}\le
     \left(1+\tfrac{h}{t}\right)^{\frac{1}{1-\alpha}} L \times\\
     &\hspace{2.5cm}\times\int_{0}^{t}
     e^{\omega(t-r)}\lnorm{\frac{F(T_{r}u_{0})-
     \left(1+\frac{h}{t}\right)^{\frac{\alpha}{\alpha-1}}
     F(T_{r+\frac{h}{t}r}u_{0})}{h}}_{X}\dr
   \end{split}
 \end{equation}
 for every $t>0$.
\end{theorem}

Before giving the proof of Theorem~\ref{thm:2bis}, we need to recall
the following definition.

\begin{definition}
  If $(X,\le)$ is an ordered vector space then a family
  $\{T_{t}\}_{t\ge 0}$ of mappings $T_{t} : C\to C$ defined on a
  subset $C\subseteq X$ is called \emph{order preserving} if
  $\{T_{t}\}_{t\ge 0}$ satisfies~\eqref{eq:23}.
\end{definition}

With this in mind, we can now give the proof above the preceding theorem.

\begin{proof}[Proof of Theorem~\ref{thm:2bis}]
  First, let $\{\tilde{T}_{t}\}_{t\ge 0}$ be the family of operators
  related to $\{T_{t}\}_{t\ge 0}$ by~\eqref{eq:47}, and for $t$, $h>0$, let
  $\lambda:=\left(1+\tfrac{h}{t}\right)$. Since $\lambda>1$, 
 $\lambda^{\frac{1}{\alpha-1}}u_{0}\le u_{0}$ if $\alpha<1$ and
 $\lambda^{\frac{1}{\alpha-1}}u_{0}\ge u_{0}$ if $\alpha>1$.
Thus, if $\alpha<1$, then by~\eqref{eq:16bis} and~\eqref{eq:23}, one has that
  \begin{align*}
    \tilde{T}_{t+h}(u_{0},\tilde{f})-\tilde{T}_{t}(u_{0},\tilde{f})
    &=\lambda^{\frac{1}{1-\alpha}}
      \tilde{T}_{t}\left[\lambda^{\frac{1}{\alpha-1}}u_{0},\lambda^{\frac{\alpha}{\alpha-1}}
      \tilde{f}(\lambda\cdot)\right]-\tilde{T}_{t}(u_{0},\tilde{f})\\
    &=\lambda^{\frac{1}{1-\alpha}}\left[
      \tilde{T}_{t}\left[\lambda^{\frac{1}{\alpha-1}}u_{0},\lambda^{\frac{\alpha}{\alpha-1}}
      \tilde{f}(\lambda\cdot)\right]-\tilde{T}_{t}\left[u_{0},\lambda^{\frac{\alpha}{\alpha-1}}
      \tilde{f}(\lambda\cdot)\right]\right]\\
    &\hspace{1.5cm}+\lambda^{\frac{1}{1-\alpha}} \tilde{T}_{t}\left[u_{0},
      \lambda^{\frac{\alpha}{\alpha-1}}\tilde{f}(\lambda\cdot)\right]-\tilde{T}_{t}(u_{0},\tilde{f})\\
    &\le \lambda^{\frac{1}{1-\alpha}}
      \left[\tilde{T}_{t}\left[u_{0},\lambda^{\frac{\alpha}{\alpha-1}}\tilde{f}(\lambda\cdot)\right]
      -\tilde{T}_{t}\left[u_{0},\tilde{f}\right]\right]\\
    &\hspace{3.5cm}+\left[\lambda^{\frac{1}{1-\alpha}}-1\right]\,\tilde{T}_{t}(u_{0},\tilde{f})
  \end{align*}
  and, similarly, if $\alpha>1$, then
  \begin{align*}
        \tilde{T}_{t+h}(u_{0},\tilde{f})-\tilde{T}_{t}(u_{0},\tilde{f})
        &\ge \lambda^{\frac{1}{1-\alpha}}
      \left[\tilde{T}_{t}\left[u_{0},\lambda^{\frac{\alpha}{\alpha-1}}\tilde{f}(\lambda\cdot)\right]
      -\tilde{T}_{t}\left[u_{0},\tilde{f}\right]\right]\\
    &\hspace{3.5cm}+\left[\lambda^{\frac{1}{1-\alpha}}-1\right]\,\tilde{T}_{t}(u_{0},\tilde{f}).
  \end{align*}
  Now, by replacing $\tilde{f}(t)$ by $-F(T_{t}u_{0})$ and by~\eqref{eq:45}, we
  can rewrite the above two inequalities and arrive to~\eqref{eq:24}
  and~\eqref{eq:13}, where $g(t)$ is given by
  \begin{displaymath}
    g_{h}(t)=\left(1+\tfrac{h}{t}\right)^{\frac{1}{1-\alpha}}\frac{
      \tilde{T}_{t}\left[u_{0},\lambda^{\frac{\alpha}{\alpha-1}}\tilde{f}(\lambda\cdot)\right]
      -\tilde{T}_{t}\left[u_{0},\tilde{f}\right]}{h}.
  \end{displaymath}
  Note, by~\eqref{eq:20bis}, one has that $g$ satisfies~\eqref{eq:15}.
\end{proof}

By Theorem~\ref{thm:2bis}, if the derivative
$\frac{\td }{\dt}_{\!+}T_{t}u_{0}$ belongs to $L^{1}_{loc}(0,T;X)$ for
$T>0$, then we can state the following.

\begin{corollary}
  \label{cor:2bis}
  Under the hypotheses of Theorem~\ref{thm:2bis}, suppose that for
  $u_{0}\in C$ satisfying $u_{0}\ge 0$, the right hand-side derivative
  $\frac{\td T_{t}u_{0}}{\dt}_{\!+} \in L^{1}_{loc}([0,T);X)$ for some
  $T>0$. Then, one has that
 \begin{displaymath}
   (\alpha-1)\frac{\td T_{t}u_{0}}{\dt}_{\!\!\!+}\ge -\frac{T_{t}u_{0}}{t}+(\alpha-1)g_{0}(t),
 \end{displaymath}
 for a.e. $t\in (0,T)$, where $g_{0} : (0,T)\to X$ is a measurable function satisfying
  \begin{equation}
    \label{eq:15bis}
     \norm{g_{0}(t)}_{X}\le \frac{L}{t} \int_{0}^{t}
      e^{\omega(t-r)}\left[\omega\lnorm{\frac{\td
            T_{r}u_{0}}{\dr}_{\!\!\!+}}_{X}
        +\frac{\abs{\alpha}}{\abs{\alpha-1}}\norm{T_{r}u}_{X}
          \right]\,\dr
  \end{equation}
 for a.e. $t\in (0,T)$.
\end{corollary}

%
%
%
%
%
%
%
%

\section{Homogeneous accretive operators}
\label{sec:semigroups}

We begin this section with the following definition. Throughout this
section, suppose $X$ is a Banach space with norm
$\norm{\cdot}_{X}$.

\begin{definition}\label{def:quasi-accretive}
  An operator $A$ on $X$ is called \emph{accretive} in $X$ if for every
  $(u,v)$, $(\hat{u},\hat{v})\in A$ and every $\lambda\ge 0$,
  \begin{displaymath}
    \norm{u-\hat{u}}_{X}\le \norm{u-\hat{u}+\lambda (v-\hat{v})}_{X}.
  \end{displaymath}
  and $A$ is called \emph{$m$-accretive} in $X$ if $A$ is accretive
  and satisfies the \emph{range condition}
  \begin{equation}
    \label{eq:11}
    Rg(I+\lambda A)=X\qquad\text{for some (or equivalently, for all)
      $\lambda>0$, $\lambda\,\omega<1$,}
  \end{equation}
  More generally, an operator $A$ on $X$ is called \emph{quasi
    ($m$-)accretive} in $X$ if there is an $\omega\in \R$
  such that $A+\omega I$ is ($m$-)accretive in $X$.
\end{definition}

If $A$ is quasi $m$-accretive in $X$, then the classical
existence theorem~\cite[Theorem~6.5]{Benilanbook}
(cf.,~\cite[Corollary~4.2]{MR2582280}) yields that for every
$u_{0}\in \overline{D(A)}^{\mbox{}_{X}}$ and $f\in L^{1}(0,T;X)$,
there is a unique \emph{mild} solution $u \in C([0,T];X)$
of~\eqref{eq:10bis}.

\begin{definition}\label{def:mild-solution}
  For given $u_{0}\in \overline{D(A)}^{\mbox{}_{X}}$ and $f\in
  L^{1}(0,T;X)$, a function $u\in C([0,T];X)$ is called a \emph{mild
    solution} of the inhomogeneous
differential inclusion~\eqref{eq:10bis} with initial value $u_{0}$
if $u(0)=u_{0}$ and for every $\varepsilon>0$,
there is a \emph{partition} $\tau_{\varepsilon} : 0=t_{0}<t_{1}<\cdots <
t_{N}=T$ and a \emph{step function}
\begin{displaymath}
  u_{\varepsilon,N}(t)=u_{0}\,\mathds{1}_{\{t=0\}}(t)+\sum_{i=1}^{N}u_{i}\,\mathds{1}_{(t_{i-1},t_{i}]}(t)
  \qquad\text{for every $t\in [0,T]$}
\end{displaymath}
satisfying
\begin{align*}
  &t_{i}-t_{i-1}<\varepsilon\qquad\text{ for all $i=1,\dots,N$,}\\
  &\sum_{N=1}^{N}\int_{t_{i-1}}^{t_{i}}\norm{f(t)-\overline{f}_{i}}\,\dt<\varepsilon\qquad
    \text{where
    $\overline{f}_{i}:=\frac{1}{t_{i}-t_{i-1}}\int_{t_{i-1}}^{t_{i}}f(t)\,\dt$,}\\
  & \frac{u_{i}-u_{i-1}}{t_{i}-t_{i-1}}+A u_{i}\ni \overline{f}_{i}\qquad\text{ for all $i=1,\dots,N$,}
\end{align*}
and
\begin{displaymath}
  \sup_{t\in [0,T]}\norm{u(t)-u_{\varepsilon,N}(t)}_{X}<\varepsilon.
\end{displaymath}
\end{definition}

Further, if $A$ is quasi $m$-accretive, then the family
$\{T_{t}\}_{t=0}^{T}$ of mappings $T_{t} :
\overline{D(A)}^{\mbox{}_{X}}\times L^{1}(0,T;X)\to
\overline{D(A)}^{\mbox{}_{X}}$ defined by~\eqref{eq:36} through the
unique mild solution $u$ of Cauchy problem~\eqref{eq:10bis} belongs to
the following class.

\begin{definition}\label{def:semigroup}
   Given a subset $C$ of $X$, a family $\{T_{t}\}_{t=0}^{T}$ of mapping $T_{t} : C\times L^{1}(0,T;X)\to
  C$ is called a \emph{strongly continuous semigroup of
    quasi-contractive mappings $T_{t}$} if $\{T_{t}\}_{t=0}^{T}$
  satisfies the following three properties:
  \begin{itemize}
  \item (\emph{semigroup property}) for every $(u_{0},f)\in \overline{D(A)}^{\mbox{}_{X}}\times L^{1}(0,T;X)$,
    \begin{equation}
      \label{eq:63}
      T_{t+s}(u_{0},f)=T_{t}(T_{s}(u_{0},f),f(s+\cdot))
    \end{equation}
    for every $t$, $s\in [0,T]$ with $t+s\le T$;
  \item (\emph{strong continuity}) for every
    $(u_{0},f)\in \overline{D(A)}^{\mbox{}_{X}}\times L^{1}(0,T;X)$,
    \begin{displaymath}
      \textrm{$t\mapsto T_{t}(u_{0},f)$ belongs to $C([0,T];X)$};
    \end{displaymath}
  \item (\emph{$\omega$-quasi contractivity}) $T_{t}$
    satisfies~\eqref{eq:20bis} with $L=1$.
  \end{itemize}
\end{definition}

Taking $f\equiv 0$ and only varying $u_{0}\in
\overline{D(A)}^{\mbox{}_{X}}$, defines by
\begin{equation}
  \tag{\ref{eq:55}}
  T_{t}u_{0}=T_{t}(u_{0},0)\qquad\text{for every $t\ge 0$,}
\end{equation}
a strongly continuous semigroup  $\{T_{t}\}_{t\ge 0}$ on $\overline{D(A)}^{\mbox{}_{X}}$ of
$\omega$-quasi contractions $T_{t} : \overline{D(A)}^{\mbox{}_{X}}\to \overline{D(A)}^{\mbox{}_{X}}$.
Given a family $\{T_{t}\}_{t\ge 0}$ of $\omega$-quasi contractions
$T_{t}$ on $\overline{D(A)}^{\mbox{}_{X}}$, then the operator
\begin{equation}
  \label{eq:61}
  A_{0}:=\Bigg\{(u_{0},v)\in X\times X\Bigg\vert\;\lim_{h\downarrow
    0}\frac{T_{h}(u_{0},0)-u_{0}}{h}=v\text{ in $X$}\Bigg\}
\end{equation}
is an $\omega$-quasi accretive well-defined mapping
$A_{0} : D(A_{0})\to X$ and called the \emph{infinitesimal generator}
of $\{T_{t}\}_{t\ge 0}$. If the Banach space $X$ and its dual space
$X^{\ast}$ are both uniformly convex
(see~\cite[Proposition~4.3]{MR2582280}), then one has that
\begin{displaymath}
  -A_{0}=A^{0},
\end{displaymath}
where $A^{\circ}$ is the minimal selection of $A$ defined
by
\begin{equation}
  \label{eq:65}
  A^{\! \circ}:=\Big\{(u,v)\in
  A\,\Big\vert\big. \norm{v}_{X}=\inf_{\hat{v}\in Au}\norm{\hat{v}}_{X}\Big\}.
\end{equation}
For simplicity, we ignore the additional geometric
assumptions on the Banach space $X$, and refer to
the two families $\{T_{t}\}_{t=0}^{T}$ defined by~\eqref{eq:36} on
$\overline{D(A)}^{\mbox{}_{X}}\times L^{1}(0,T;X)$ and
$\{T_{t}\}_{t\ge 0}$ defined by~\eqref{eq:55} on
$\overline{D(A)}^{\mbox{}_{X}}$ as the \emph{semigroup generated by $-A$}.\medskip

Further, for every
$u_{0}\in \overline{D(A)}^{\mbox{}_{X}}$, if $f\in L^{1}(0,T;X)$ is given by the step function
$f=\sum_{i=1}^{N}f_{i}\,\mathds{1}_{(t_{i-1}, t_{i}]}$, then the
corresponding mild solution $u : [0,T]\to X$ of Cauchy problem
\eqref{eq:10bis} is given by
\begin{equation}
  \label{eq:51}
  u(t)= u_{0}\,\mathds{1}_{\{t=0\}}(t)+\sum_{i=1}^{N}u_{i}(t) \mathds{1}_{(t_{i-1},t_{i}]}(t)
\end{equation}
where each $u_{i}$ is the unique mild solution of the Cauchy problem
(for constant $f\equiv f_{i}$)
\begin{equation}
  \label{eq:52}
  \frac{\td u_{i}}{\dt}+A(u_{i}(t))\ni f_{i}\quad\text{ on
    $(t_{i-1},t_{i})$, and }\quad u_{i}(t_{i-1})=u_{i-1}(t_{i-1})
\end{equation}
for every $i=1,\dots, N$ (cf.,~\cite[Chapter~4.3]{Benilanbook}). In particular,
the semigroup $\{T_{t}\}_{t=0}^{T}$ is
obtained by the \emph{exponential formula}
\begin{equation}
  \label{eq:27}
  T_{t}(u(t_{i-1}),f_{i})=u_{i}(t)=\lim_{n\to\infty} \left[J_{\frac{t-t_{i-1}}{n}}^{A-f_{i}}\right]^{n}u(t_{i-1})\qquad\text{in
  $C([t_{i-1},t_{i}];X)$}
\end{equation}
iteratively for every $i=1,\dots,N$, where for $\mu>0$,
$J_{\mu}^{A-f_{i}}=(I+\mu (A-f_{i}))^{-1}$ is the \emph{resolvent operator}
of $A-f_{i}$.\medskip

As for classical solutions, the fact that $A$ is homogeneous
  of order $\alpha\neq 1$, is also reflected in the notion of mild
  solution and, in particular, in the semigroup $\{T_{t}\}_{t=0}^{T}$
  as demonstrated in our next proposition.

\begin{proposition}[{\bfseries Homogeneous accretive operators}]\label{acrethomog} 
  Let $A$ be a quasi $m$-accretive operator on $X$ and
  $\{T_{t}\}_{t=0}^{T}$ the semigroup generated by $-A$ on
  $\overline{D(A)}^{\mbox{}_{X}}\times L^{1}(0,T;X)$. If $A$ is
  homogeneous of order $\alpha\neq 1$, then for every $\lambda>0$, $\{T_{t}\}_{t=0}^{T}$
  satisfies equation
\begin{equation}
 \tag{\ref{eq:21bis}}
  \lambda^{\frac{1}{\alpha-1}}T_{\lambda
    t}(u_{0},f)=T_{t}(\lambda^{\frac{1}{\alpha-1}}u_{0},\lambda^{\frac{\alpha}{\alpha-1}}f(\lambda\cdot))\qquad
  \text{for all $t\in \left[0,\tfrac{T}{\lambda}\right]$,}
\end{equation}
for every $(u_{0},f)\in \overline{D(A)}^{\mbox{}_{X}}\times L^{1}(0,T;X)$.
\end{proposition}

\begin{proof}
 Let $\lambda>0$ and $f\in X$. Then, for every $u$, $v\in X$ and $\mu>0$, 
  \begin{displaymath}
    J_{\mu}^{A-\lambda^{\frac{\alpha}{\alpha-1}}f}\left[\lambda^{\frac{1}{\alpha-1}}v\right]=u\qquad\text{if and only if }\qquad
    u+\mu (Au-\lambda^{\frac{\alpha}{\alpha-1}} f)\ni \lambda^{\frac{1}{\alpha-1}}v.
  \end{displaymath}
  Now, the hypothesis that $A$ is \emph{homogeneous of
    order $\alpha\neq 1$} implies that the right-hand side in the previous characterization is equivalent to
  \begin{displaymath}
    \lambda^{\frac{1}{1-\alpha}} u+\lambda\mu
    (A(\lambda^{\frac{1}{1-\alpha}} u)-f)\ni v,   \qquad\text{or }\qquad
    J_{\lambda\mu}^{A-f}v=\lambda^{\frac{1}{1-\alpha}} u.
  \end{displaymath}
  Therefore, one has that
  \begin{equation}
    \label{eq:26}
    \lambda^{\frac{1}{\alpha-1}}\, J^{A-f}_{\lambda\mu}v=
    J_{\mu}^{A-\lambda^{\frac{\alpha}{\alpha-1}}f}\left[\lambda^{\frac{1}{\alpha-1}}v\right]\qquad\text{for all $\lambda$,
      $\mu>0$, and $v\in X$.}
  \end{equation}
 Now, let $u_{0}\in
 \overline{D(A)}^{\mbox{}_{X}}$, $\pi : 0=t_{0}<t_{1}<\cdots <
 t_{N}=T$ be a partition of $[0,T]$, and $f=\sum_{i=1}^{N}f_{i}\mathds{1}_{(t_{i-1},t_{i}]}\in L^{1}(0,T;X)$
a step function. If $u$ denotes the unique mild solution
of~\eqref{eq:10bis} for this step function $f$, then $u$ is given by~\eqref{eq:51}, were
on each subinterval $(t_{i-1},t_{i}]$, $u_{i}$ is the unique mild
solution of~\eqref{eq:52}. 

Next, let $\lambda>0$ and set
\begin{displaymath}
v_{\lambda}(t):=\lambda^{\frac{1}{\alpha-1}}u(\lambda t)\qquad\text{for
every $t\in \left[0,\tfrac{T}{\lambda}\right]$.}
\end{displaymath}
Then,
\begin{displaymath}
  v_{\lambda}(t)=
  \lambda^{\frac{1}{\alpha-1}}u_{0}\,\mathds{1}_{\{t=0\}}(t)+\sum_{i=1}^{N}\lambda^{\frac{1}{\alpha-1}}u_{i}(\lambda
  t) \mathds{1}_{\left(\frac{t_{i-1}}{\lambda},\frac{t_{i}}{\lambda}\right]}(t)
\end{displaymath}
for every $t\in \left[0,\tfrac{T}{\lambda}\right]$. Obviously,
$v_{\lambda}(0)=\lambda^{\frac{1}{\alpha-1}}u_{0}$. Thus, to show that~\eqref{eq:21bis} holds, it
remains to verify that $v_{\lambda}$ is a mild solution of
\begin{displaymath}
  \tfrac{\td v_{\lambda}}{\dt}+A(v_{\lambda}(t))\ni
  \lambda^{\frac{\alpha}{\alpha-1}}f(\lambda t)\qquad\text{on $\left(0,\tfrac{T}{\lambda}\right)$}
\end{displaymath}
or, in other words, 
\begin{equation}
  \label{eq:46}
  v_{\lambda}(t)=T_{t}(\lambda^{\frac{1}{\alpha-1}}u_{0},\lambda^{\frac{\alpha}{\alpha-1}}f(\lambda\cdot))
\end{equation}
for every $t\in \left[0,\tfrac{T}{\lambda}\right]$. Let $t\in (0,t_{1}/\lambda]$ and $n\in \N$. We apply~\eqref{eq:26} to
\begin{displaymath}
  \mu=\frac{t}{n}\qquad\text{and}\qquad
  v=J^{A-\lambda^{\frac{\alpha}{\alpha-1}}f_{1}}_{\frac{\lambda t}{n}}[\lambda^{\frac{1}{\alpha-1}}u_{0}].
\end{displaymath}
Then, one finds that
\begin{displaymath}
  \left[J^{A-\lambda^{\frac{\alpha}{\alpha-1}}f_{1}}_{\frac{t}{n}}\right]^{2}[\lambda^{\frac{1}{\alpha-1}}u_{0}]
=  J_{\frac{t}{n}}^{A-\lambda^{\frac{\alpha}{\alpha-1}}f_{1}}\left[\lambda^{\frac{1}{\alpha-1}}J_{\frac{\lambda
        t}{n}}^{A-f_{1}}u_{0}\right]
  =\lambda^{\frac{1}{\alpha-1}}\left[J^{A-f_{1}}_{\frac{\lambda t}{n}}\right]^{2}u_{0}.
\end{displaymath}
Applying~\eqref{eq:26} to
$\lambda^{\frac{1}{\alpha-1}}\left[J^{A-f_{1}}_{\frac{\lambda
      t}{n}}\right]^{i}u_{0}$ iteratively for $i=2,\dots, n$ yields
\begin{equation}
  \label{eq:37}
  \lambda^{\frac{1}{\alpha-1}}\, \left[J^{A-f_{1}}_{\frac{\lambda  t}{n}}\right]^{n}u_{0}=
\left[J^{A-\lambda^{\frac{\alpha}{\alpha-1}}f_{1}}_{\frac{t}{n}}\right]^{n}\left[\lambda^{\frac{1}{\alpha-1}}u_{0}\right].
\end{equation}
By~\eqref{eq:27}, sending $n\to +\infty$ in~\eqref{eq:37} yields on the one side
\begin{displaymath}
  \lim_{n\to+\infty}
  \lambda^{\frac{1}{\alpha-1}}\,
  \left[J^{A-f_{1}}_{\frac{\lambda
        t}{n}}\right]^{n}u_{0}=\lambda^{\frac{1}{\alpha-1}}\,u_{1}(\lambda
  t)=v_{\lambda}(t),
\end{displaymath}
and on the other side
\begin{displaymath}
  \lim_{n\to+\infty}\left[
    J^{A-\lambda^{\frac{\alpha}{\alpha-1}}f_{1}}_{\frac{t}{n}}\right]^{n}\left[\lambda^{\frac{1}{\alpha-1}}u_{0}\right]
  =T_{t}(\lambda^{\frac{1}{\alpha-1}}u_{0},\lambda^{\frac{\alpha}{\alpha-1}}f_{1}),
\end{displaymath}
showing that~\eqref{eq:46} holds for every $t\in [0,\frac{t_{1}}{\lambda}]$. 
Repeating this argument on each subinterval $(\frac{t_{i-1}}{\lambda},
\frac{t_{i}}{\lambda}]$ for $i=2,\dots, N$, where one replaces in~\eqref{eq:37} $u_{0}$
by $u(t_{i-1})$, and $f_{1}$ by $f_{i}$, then one sees that $v_{\lambda}$
satisfies~\eqref{eq:46} on the whole interval $[0,\frac{T}{\lambda}]$.
 \end{proof}

 By the preceding proposition and by Lemma~\ref{thm:1bis}, we can now state the
 following result.

\begin{corollary}\label{cor:1bis}
  Let $A$ be a quasi $m$-accretive
  operator on a Banach space $X$ and $\{T_{t}\}_{t=0}^{T}$ the
  semigroup generated by $-A$ on $L^{1}(0,T;X)\times
  \overline{D(A)}^{\mbox{}_{X}}$. If $A$ is
  homogeneous of order $\alpha\neq 1$, then for every
  $(u_{0},f)\in \overline{D(A)}^{\mbox{}_{X}}\times L^{1}(0,T;X)$,
  $t\mapsto T_{t}(u_{0},f)$ satisfies
  \begin{equation}
    \label{eq:22}
    \begin{split}
        &\limsup_{h\to 0+}\lnorm{\frac{T_{t+h}(u_{0},f)-T_{t}(u_{0},f)}{h}}_{X}\\
        & \qquad \le \frac{1}{t}\,e^{\omega t}
        \left[2\frac{\norm{u_{0}}_{X}}{\abs{1-\alpha}}\,
          +\frac{1}{\abs{1-\alpha}}\int_{0}^{t}e^{-\omega
            s}\norm{f(s)}_{X}\,\ds+V_{\omega}(f,t)\right],
      \end{split}
  \end{equation}
  for a.e. $t\in (0,T]$, where $V_{\omega}(f,t)$ is defined
  by~\eqref{eq:64}. In particular, if $f\in W^{1,1}(0,T;X)$ and
  $\tfrac{\td}{\dt}T_{t}(u_{0},f)$ exists in $X$ at a.e. $t\in (0,T)$,
  then $T_{t}(u_{0},f)$ satisfies
  \begin{equation}
    \label{eq:32}
    \begin{split}
        \lnorm{\frac{\td }{\dt}_{\!\! +}\! T_{t}(u_{0},f)}_{X} & \le
        \frac{L}{t}\,e^{\omega t}
        \left[2\frac{\norm{u_{0}}_{X}}{\abs{1-\alpha}}\,
          +\frac{1}{\abs{1-\alpha}}\int_{0}^{t}e^{-\omega
            s}\norm{f(s)}_{X}\,\ds\right.\\
        &\left.\hspace{4.4cm}+\int_{0}^{t}e^{-\omega
            s}\norm{f'(s)}_{X}\,s\,\ds \right]
      \end{split}
  \end{equation}
  for a.e. $t\in (0,T)$.
\end{corollary}

To consider the regularizing effect of mild solutions to the
Cauchy problem~\eqref{eq:18} for the perturbed operator $A+F$, we
recall the following well-known result from the literature.

\begin{proposition}[{\cite[Lemma~7.8]{Benilanbook}}]\label{prop:Lipschitz-property}
  If $A+\omega I$ is accretive in $X$ and $f\in BV(0,T;X)$, then for
  every $u_{0}\in D(A)$, the mild
  solution $u(t):=T_{t}(u_{0},f)$, ($t\in [0,T]$), of Cauchy problem~\eqref{eq:10bis} is
  Lipschitz continuous on $[0,T]$ satisfying
  \begin{align*}
    &\limsup_{h\to 0+}\frac{\norm{u(t+h)-u(t)}_{X}}{h}\\
    &\qquad\le
    e^{\omega t}\norm{f(0+)-y}_{X}+\tilde{V}(f,t+)
    +\omega\int_{0}^{t}e^{\omega(t-s)}\tilde{V}(f,s+)\,\ds 
  \end{align*}
  for every $t\in [0,T]$ and $v\in Au_{0}$, where
  \begin{displaymath}
    \tilde{V}(f,t+):=\limsup_{h\to
      0+}\int_{0}^{t}\frac{\norm{f(s+h)-f(s)}_{X}}{h}\,\ds.
  \end{displaymath}
\end{proposition}

With the preceding Theorem~\ref{cor:1bisF},
Proposition~\ref{acrethomog}, and Proposition~\ref{prop:Lipschitz-property} in mind, we
are now in the position to outline the proof of our main Theorem~\ref{thm:main1}.

\begin{proof}[Proof of Theorem~\ref{thm:main1}]
  We begin by noting that if $A$ is $m$-accretive in $X$ and $F$ is a
  Lipschitz continuous mapping with Lipschitz constant $\omega$, then
  the operator $A+F$ is $\omega$-quasi $m$-accretive in $X$; or in
  other words, $A+F+\omega I$ is $m$-accretive in $X$. Hence, for
  every $T>0$, there is a semigroup $\{T_{t}\}_{t=0}^{T}$ of
  mappings
  $T_{t}:\overline{D(A)}^{\mbox{}_{X}}\times L^{1}(0,T;X)\to
  \overline{D(A)}^{\mbox{}_{X}}$ satisfying~\eqref{eq:4}
  and~\eqref{eq:20bis} with $\omega$ and $L=1$. Further, the semigroup
  $\{\tilde{T}_{t}\}_{t=0}^{T}$ generated by $-A$ satisfies
  \eqref{eq:4} and~\eqref{eq:20bis} with $\omega=0$ and $L=1$,
  \eqref{eq:45} for every $u_{0}\in \overline{D(A)}^{\mbox{}_{X}}$ and
  $f\in L^{1}(0,T;X)$ with $\tilde{f}$ given by~\eqref{eq:43}, and by
  Proposition~\ref{acrethomog}, $\{\tilde{T}_{t}\}_{t=0}^{T}$
  satisfies \eqref{eq:21bis}. Now, let $u_{0}\in D(A)$ and
  $f\in BV(0,T;X)$. Then by Proposition~\ref{prop:Lipschitz-property},
  the mild solution $u(t):=T_{t}(u_{0},f)$, ($t\in [0,T]$), of Cauchy
  problem~\eqref{eq:10bis} is Lipschitz continuous on $[0,T]$. Thus,
  we can apply Theorem~\ref{cor:1bisF} and obtain that $u$
  satisfies~\eqref{eq:28}.
\end{proof}

In order the semigroup
$\{T_{t}\}_{t=0}^{T}$ generated by $-A$ satisfies regularity estimate~\eqref{eq:40}
(respectively,~\eqref{eq:18bis2}), one requires that each mild solution $u$
of~\eqref{eq:10bis} (respectively, of~\eqref{eq:10}) is
\emph{differentiable} at a.e. $t\in (0,T)$, or in other words, $u$ is
a \emph{strong solution} of~\eqref{eq:10bis}. The next definition is taken
from~\cite[Definition~1.2]{Benilanbook}
(cf~\cite[Chapter~4]{MR2582280}).

\begin{definition}\label{def:strong-sols}
  A locally absolutely continuous function $u [0,T] : \to X$ is called a \emph{strong
      solution} of the differential inclusion
    \begin{equation}
      \label{eq:31}
      \frac{\td u}{\dt}(t)+A(u(t))\ni f(t) \qquad\text{for
        a.e. $t\in (0,T)$,}
    \end{equation}
  if $u$ is differentiable a.e. on $(0,T)$, and for a.e. $t\in (0,T)$,
    $u(t)\in D(A)$ and $f(t)-\frac{\td u}{\dt}(t)\in
    A(u(t))$. Further, for given $u_{0}\in X$ and $f\in L^{1}(0,T;X)$,
    a function $u$ is called a \emph{strong
      solution} of Cauchy problem~\eqref{eq:10bis} if $u\in
    C([0,T];X)$, $u$ is strong solution of~\eqref{eq:31} and $u(0)=u_{0}$.
\end{definition}

The next characterization of strong solutions of~\eqref{eq:31}
highlights the important point of \emph{a.e. differentiability}.

\begin{proposition}[{\cite[Theorem~7.1]{Benilanbook}}]\label{propo:diff-mild}
  Let $X$ be a Banach space, $f\in L^{1}(0,T;X)$ and $A$ be
  quasi $m$-accretive in $X$. Then $u$ is a strong solution
  of the differential inclusion~\eqref{eq:31} on $[0,T]$ if and only
  if $u$ is a mild solution on $[0,T]$ and $u$ is ``absolutely
  continuous'' on $[0,T]$ and differentiable a.e. on $(0,T)$.
\end{proposition}

Of course, every strong solution $u$ of~\eqref{eq:31} is a
  mild solution of~\eqref{eq:31}, and $u$ is absolutely continuous and
  differentiable a.e. on $[0,T]$. The differential
  inclusion~\eqref{eq:31} admits mild and Lipschitz continuous
  solutions if $A$ is $\omega$-quasi $m$-accretive in $X$
  (cf~\cite[Lemma~7.8]{Benilanbook}). But, in general, an absolutely continuous
  functions $u : [0,T]\to X$ is not necessarily
  differentiable a.e. on $(0,T)$. Only if one assumes additional
  geometric properties on $X$, then the latter implication holds true. Our next
definition is taken from~\cite[Definition~7.6]{Benilanbook}
(cf~\cite[Chapter~1]{MR2798103}).

\begin{definition}\label{def:Radon}
  A Banach space $X$ is said to have the \emph{Radon-Nikod\'ym property}
  if every absolutely continuous function $F : [a,b]\to X$, ($a$,
  $b\in \R$, $a<b$), is differentiable almost everywhere on $(a,b)$.
\end{definition}

Known examples of Banach spaces $X$ admitting the Radon-Nikod\'ym property are:
\begin{itemize}
\item {\bfseries (Dunford-Pettis)} if $X=Y^{\ast}$ is separable, where
  $Y^{\ast}$ is the dual space of a Banach space $Y$;
\item if $X$ is \emph{reflexive}.
\end{itemize}

We emphasize that $X_{1}=L^{1}(\Sigma,\mu)$,
$X_{2}=L^{\infty}(\Sigma,\mu)$, or $X_{3}=C(\mathcal{M})$ for a
$\sigma$-finite measure space $(\Sigma,\mu)$, or respectively, for a
compact metric space $(\mathcal{M},d)$ don't have, in general, the
Radon-Nikod\'ym property (cf~\cite{MR2798103}). Thus, it is quite
surprising that there is a class of operators $A$ (namely, the class of
\emph{completely accretive operators}, see
Section~\ref{sec:completely-accretive} below), for which the
differential inclusion~\eqref{eq:10bis} nevertheless admits strong
solutions (with values in $L^{1}(\Sigma,\mu)$ or $L^{\infty}(\Sigma,\mu)$).\bigskip

Now, by Corollary~\ref{cor:1bis} and
Proposition~\ref{propo:diff-mild}, we can conclude the following
results. We emphasize that one crucial point in the statement of
Corollary~\ref{cor:RN} below is that due to the uniform
estimate~\eqref{eq:39}, one has that for all initial values
$u_{0}\in \overline{D(A)}^{\mbox{}_{X}}$, the unique mild solution
$u$ of~\eqref{eq:10bis} is strong. 

\begin{corollary}\label{cor:RN}
  Suppose $A$ is a quasi $m$-accretive operator on a Banach space $X$
  admitting the Radon-Nikod\'ym property, and $\{T_{t}\}_{t=0}^{T}$ is
  the semigroup generated by $-A$ on
  $\overline{D(A)}^{\mbox{}_{X}}\times L^{1}(0,T;X)$. If $A$ is
  homogeneous of order $\alpha\neq 1$, then for every
  $u_{0}\in \overline{D(A)}^{\mbox{}_{X}}$ and $f\in W^{1,1}(0,T;X)$,
  the unique mild solution $u(t):=T_{t}(u_{0},f)$ of~\eqref{eq:10bis}
  is strong and $\{T_{t}\}_{t=0}^{T}$ satisfies~\eqref{eq:32} for
  a.e. $t\in (0,T)$.
\end{corollary}

We omit the proof of Corollary~\ref{cor:RN} since it is straightforward.
Now, we are ready to show that the statement of
Corollary~\ref{cor:RN-Lipschitz-case} holds.

\begin{proof}[Proof of Corollary~\ref{cor:RN-Lipschitz-case}]
  First, let $u_{0}\in D(A)$. Then by
  Proposition~\ref{prop:Lipschitz-property}, the mild solution
  $u(t)=T_{t}(u_{0},f)$ of Cauchy problem~\eqref{eq:10bisF} is Lipschitz continuous on $[0,T)$, and since
  every reflexive Banach space $X$ admits the Radon-Nikod\'ym
  property, $u$ is differentiable a.e. on $(0,T)$. Thus,
  by Theorem~\ref{thm:main1}, there is a function $\psi\in
  L^{\infty}(0,T)$ such that $u$ satisfies
  \begin{displaymath}
    \lnorm{\frac{\td u}{\dt}_{\!\! +}\!\!(t)}_{X}
    \le\frac{1}{t}\!\! \left[\frac{e^{\omega
          t}+1}{\abs{1-\alpha}}\norm{u_{0}}_{X}+\psi(t)
      + \omega\!\!\int_{0}^{t}\left[\frac{e^{\omega s}+1}{\abs{1-\alpha}}\norm{u_{0}}_{X}+\psi(s)\right]
      e^{\omega (t-s)}\ds\right]
  \end{displaymath}
  for a.e. $t\in (0,T)$. Next, we square
  both sides of the last inequality and subsequently integrate over $(a,b)$ for given
  $0<a<b\le T$. Then, one finds that
  \begin{equation}
    \label{eq:59}
    \begin{split}
      \int_{a}^{b}\lnorm{\frac{\td u}{\dt}(t)}_{X}^{2}\dt&\le
      \int_{a}^{b}\frac{1}{t^2}\!\!  \left\{\frac{e^{\omega
            t}+1}{\abs{1-\alpha}}\norm{u_{0}}_{X}+\psi(t)\right. \\
    &\qquad \qquad\left. +
      \omega\!\!\int_{0}^{t}\left[\frac{e^{\omega
            s}+1}{\abs{1-\alpha}}\norm{u_{0}}_{X}+\psi(s)\right]
      e^{\omega (t-s)}\ds\right\}^2 \!\dt
  \end{split}
  \end{equation}
  Due to this estimate, we can now show that also for $u_{0}\in
  \overline{D(A)}^{\mbox{}_{X}}$, the corresponding mild solution $u$
  of~\eqref{eq:10bisF} is strong. To see this let $(u_{0,n})_{n\ge 1}\subseteq D(A)$ such that $u_{0,n}\to u_{0}$
  in $X$ as $n\to \infty$ and set $u_{n}=T_{t}u_{0,n}$ for all $n\ge
  1$. By~\eqref{eq:20bis} (which is
  satisfied with $L=1$ by all $u_{n}$), $(u_{n})_{n\ge 1}$ is a Cauchy
  sequence in $C([0,T];X)$. Hence, there is a function $u\in
  C([0,T];X)$ satisfying $u(0)=u_{0}$ and $u_{n}\to u$ in
  $C([0,T];X)$. In particular, one can show that $u$ is the unique
  mild solution of Cauchy problem~\eqref{eq:10bisF}. Now, by inequality~\eqref{eq:59},
  $\left(\td u_{n}/\dt\right)_{n\ge 1}$ is bounded in
  $L^{2}(a,b;X)$ for any $0<a<b\le T$. Since $X$ is reflexive, also $L^{2}(a,b;X)$ is
  reflexive and hence, there is a $v\in L^{2}(a,b;X)$ and a
  subsequence of $(u_{0,n})_{n\ge 1}$, which, for
  simplicity, we denote again by $(u_{0,n})_{n\ge 1}$, such that
  $\frac{\td u_{n}}{\dt}\rightharpoonup v$ weakly in
  $L^{2}(a,b;X)$ as $n\to +\infty$. Since $u_{n}\to u$ in
  $C([a,b];X)$, it follows by a standard argument that
  $v=\frac{\td u}{\dt}$ in the sense of vector-valued
  distributions $\mathcal{D}'((a,b);X)$. Since
  $\frac{\td u}{\dt}\in L^{2}(a,b;X)$, the mild solution
  $u$ of Cauchy problem~\eqref{eq:10bisF} is absolutely continuous on $(a,b)$, and since $X$
  is reflexive, $u$ is differentiable a.e. on $(a,b)$. Since
  $0<a<b<\infty$ were arbitrary,
  $\frac{\td u}{\dt}\in L^{1}_{loc}((0,\infty);X)$. 

  To see that $u$ satisfies inequality~\eqref{eq:29}, note that 
  for $\varepsilon>0$, the function $t\mapsto \tilde{u}(t):=u(t+\varepsilon)$
  on $[0,T-\varepsilon]$ satisfies the hypotheses of Theorem~\ref{thm:main1} with
  $\frac{\td \tilde{u}}{\dt}\in
  L^{1}([0,T-\varepsilon);X)$ and so, we find that
  \begin{displaymath}
    \begin{split}
\lnorm{\frac{\td \tilde{u}}{\dt}_{\!\! +}\!\!(t)}_{X}
    &\le\frac{1}{t}\!\! \left[\frac{e^{\omega
          t}+1}{\abs{1-\alpha}}\norm{u(\varepsilon)}_{X}+\psi(t)\right.\\
    &  \qquad \qquad\left.+ \omega\!\!\int_{0}^{t}
      \left[\frac{e^{\omega s}+1}{\abs{1-\alpha}}\norm{u(\varepsilon)}_{X}+\psi(s)\right]
      e^{\omega (t-s)}\ds\right]
    \end{split}
  \end{displaymath}
  for every $t\in (0,T-\varepsilon]$ and $\varepsilon\in (0,t)$. Sending $\varepsilon\to
  0+$ shows that $u$ satisfies~\eqref{eq:29}. Since
  $u_{0}\in \overline{D(A)}^{\mbox{}_{X}}$ was arbitrary, this
  completes the proof of this corollary.
\end{proof}

If the Banach space $X$ and its dual space $X^{\ast}$ are
\emph{uniformly convex} and $A+F$ is quasi $m$-accretive in $X$, then
(cf.,~\cite[Theorem~4.6]{MR2582280}) for every $u_{0}\in D(A)$,
$f\in W^{1,1}(0,T;X)$, the mild solution $u(t)=T_{t}(u_{0},f)$,
($t\in [0,T]$), of Cauchy problem~\eqref{eq:10bisF} is a strong one, $u$ is everywhere
differentiable from the right, $\frac{\td u}{\dt}_{\! +}$ is right
continuous, and
\begin{displaymath}
  \frac{\td u}{\dt}_{\!+}(t)+(A+F-f(t))^{\! \circ}u(t)=0\qquad\text{for
    every $t\ge 0$,}
\end{displaymath}
where for every $t\in [0,T]$, $(A+F-f(t))^{\! \circ}$ denotes the
\emph{minimal selection} of $A+F-f(t)$ defined by~\eqref{eq:65}. Thus,
under those assumptions on $X$ and by
Corollary~\ref{cor:RN-Lipschitz-case}, we can conclude the following two
corollaries. We begin by stating the inhomogeneous case.

\begin{corollary}\label{cor:Lipschitz-continuous-case}
  Suppose $X$ and its dual space $X^{\ast}$ are
  uniformly convex, for $\omega\in \R$, $A$ is an
  $\omega$-quasi $m$-accretive operator on $X$, and
  $\{T_{t}\}_{t=0}^{T}$ is the semigroup on
  $\overline{D(A)}^{\mbox{}_{X}}\times L^{1}(0,T;X)$ generated by $-A$. If $A$ is
  homogeneous of order $\alpha\neq 1$, then for every
  $u_{0}\in \overline{D(A)}^{\mbox{}_{X}}$ and $f\in W^{1,1}(0,T;X)$, the mild solution
  $u(t)=T_{t}(u_{0},f)$, ($t\in [0,T]$) of Cauchy problem~\eqref{eq:10bis} is
  strong and
  \begin{displaymath}
    \begin{split}
      \lnorm{(A+F-f(t))^{\! \circ}T_{t}(u_{0},f)}_{X}
    \le\frac{1}{t}\!\! \left[a(t)+ \omega\!\!\int_{0}^{t}a(s)
      e^{\omega (t-s)}\ds\right]
    \end{split}
    \end{displaymath}
    for every $t\in (0,T]$, where $a(t)$ is defined by~\eqref{eq:54}.
\end{corollary}


Our next corollary considers the homogeneous case of Cauchy problem~\eqref{eq:10bis}.

\begin{corollary}
  Suppose $X$ and its dual space $X^{\ast}$ are uniformly convex, for
  $\omega\in \R$, $A$ is an $\omega$-quasi $m$-accretive operator on
  $X$, and $\{T_{t}\}_{t\ge 0}$ is the semigroup on
  $\overline{D(A)}^{\mbox{}_{X}}$ generated by $-A$. If $A$ is
  homogeneous of order $\alpha\neq 1$, then for every
  $u_{0}\in \overline{D(A)}^{\mbox{}_{X}}$, the mild solution
  $u(t)=T_{t}u_{0}$, ($t\ge 0$) of Cauchy problem~\eqref{eq:10} (for
  $f\equiv 0$) is a strong solution satisfying
  \begin{displaymath}
    \lnorm{(A+F-f(t))^{\! \circ}T_{t}u_{0}}_{X}
    \le\frac{e^{\omega t}+1}{\abs{1-\alpha}\,t}
    \left[1+\omega\int_{0}^{t}e^{\omega(t-s)}\,\ds\right]\, \norm{u_{0}}_{X}
  \end{displaymath}
  for every $t>0$.
\end{corollary}



To conclude this section, we briefly outline the proof of Theorem~\ref{thm:2}.

\begin{proof}[Proof of Theorem~\ref{thm:2}]
  In the case $A$ is an $m$-accretive operator on a Banach lattice
  $X$ and $F$ a Lipschitz continuous perturbation with constant
  $\omega\ge 0$, then the statement of Theorem~\ref{thm:2} immediately
  follow  from Theorem~\ref{thm:2bis} and Corollary~\ref{cor:2bis}
  with constants $L=1$.
\end{proof}

%
%
\section{Homogeneous completely accretive operators}
\label{sec:completely-accretive}

In \cite{MR1164641}, B\'enilan and Crandall introduced the class of
\emph{completely accretive} operators $A$ and showed: even though the
underlying Banach spaces does not admit the Radon-Nikod\'ym property,
but if $A$ is completely accretive and homogeneous of order $\alpha>0$
with $\alpha\neq 1$, then the mild solutions of differential
inclusion~\eqref{eq:10} involving $A$ are strong. This was extended
in~\cite{MR4031770} to the zero-order case. Here, we provide a
generalization to the case of completely accretive operators which are
homogeneous of order $\alpha\neq 1$ and perturbed by a Lipschitz
nonlinearity.
%

\subsection{General framework}
\label{subsec:gen-framework}

In order to keep this paper self-contained, we provide a brief
introduction to the class of completely accretive operators, where 
we mainly follow~\cite{MR1164641}  and the monograph~\cite{CoulHau2016}.\medskip

For the rest of this paper, suppose $(\Sigma, \mathcal{B}, \mu)$ is a $\sigma$-finite
measure space, and $M(\Sigma,\mu)$ the space of $\mu$-a.e. equivalent
classes of measurable functions $u : \Sigma\to \R$. For
$u\in M(\Sigma,\mu)$, we write $[u]^+$ to denote $\max\{u,0\}$ and
$[u]^-=-\min\{u,0\}$.  We denote by
$L^{q}(\Sigma,\mu)$, $1\le q\le \infty$, the corresponding standard
Lebesgue space with norm
\begin{displaymath}
  \norm{\cdot}_{q}=
  \begin{cases}
    \displaystyle\left(\int_{\Sigma}\abs{u}^{q}\,\textrm{d}\mu\right)^{1/q} &
    \text{if $1\le q<\infty$,}\\[7pt]
    \inf\Big\{k\in [0,+\infty]\;\Big\vert\;\abs{u}\le k\text{
      $\mu$-a.e. on $\Sigma$}\Big\}
    & \text{if $q=\infty$.}
  \end{cases}
\end{displaymath}
For $1\le q<\infty$, we
identify the dual space $(L^{q}(\Sigma,\mu))'$ with
$L^{q^{\mbox{}_{\prime}}}(\Sigma,\mu)$,
where $q^{\mbox{}_{\prime}}$ is the conjugate exponent of $q$ given by
$1=\tfrac{1}{q}+\tfrac{1}{q^{\mbox{}_{\prime}}}$.\medskip

Next, we first briefly recall the notion of \emph{Orlicz spaces}
(cf~\cite[Chapter 3]{MR1113700}). A continuous function
$\psi : [0,+\infty) \to [0,+\infty)$ is an \emph{$N$-function} if it
is convex, $\psi(s)=0$ if and only if $s=0$,
$\lim_{s\to 0+} \psi (s) / s = 0$, and
$\lim_{s\to\infty} \psi (s) / s = \infty$. Given an $N$-function
$\psi$, the \emph{Orlicz space} is defined as follows
  \begin{displaymath}
    L^\psi (\Sigma,\mu) := \Bigg\{ u \in M(\Sigma,\mu) \Bigg\vert\;
    \int_\Sigma
    \psi\left(\frac{|u|}{\alpha}\right)\; \td\mu <\infty \text{ for some }
    \alpha >0\Bigg\}
  \end{displaymath}
  and equipped with the \emph{Orlicz-Minkowski norm}
  \begin{displaymath}
    \norm{u}_{\psi} := \inf \Bigg\{ \alpha >0 \;\Bigg\vert\; \int_\Sigma \psi
    \left(\frac{|u|}{\alpha}\right)\; \td\mu\le 1 \Bigg\} .
  \end{displaymath}

With these preliminaries in mind, we are now in the position to recall the notation of
\emph{completely accretive} operators introduced in \cite{MR1164641} and
further developed to the \emph{$\omega$-quasi} case in~\cite{CoulHau2016}.\medskip

Let $J_{0}$ be the set given by
  \begin{displaymath}
    J_0 = \Big\{ j : \R \rightarrow
    [0,+\infty]\;\Big\vert\big. \text{$j$ is convex, lower
      semicontinuous, }j(0) = 0 \Big\}.
  \end{displaymath}
  Then, for every $u$, $v\in M(\Sigma,\mu)$, we write
  \begin{displaymath}
  u\ll v \quad \text{if and only if}\quad \int_{\Sigma} j(u)
  \,\td\mu \le \int_{\Sigma} j(v) \, \td\mu\quad\text{for all $j\in J_{0}$.}
\end{displaymath}

\begin{remark}
  Due to the interpolation result~\cite[Proposition~1.2]{MR1164641},
  for given $u$, $v\in M(\Sigma,\mu)$, the relation $u\ll v$ is
  equivalent to the two conditions
  \begin{displaymath}
    \begin{cases}
      \displaystyle\int_{\Sigma}[u-k]^{+}\,\td\mu&\le
      \int_{\Sigma}[v-k]^{+}\,\td\mu\qquad\text{for all $k>0$ and}\\[7pt]
       \displaystyle \int_{\Sigma}[u+k]^{-}\,\td\mu&\le
      \int_{\Sigma}[v+k]^{-}\,\td\mu\qquad\text{for all $k>0$.}
    \end{cases}
  \end{displaymath}
  By this characterization, it is clear that for every $u$, $v$, $w\in
  M(\Sigma,\mu)$, 
  \begin{equation}
    \label{eq:66}
    \text{if $u\ll v$ and $0\le w\le u$ then $w\ll v$.}
  \end{equation}
  Thus, the relation $\ll$ is closely related to the theory of
  rearrangement-invariant function spaces
  (cf~\cite{MR928802}). Another, useful characterization of the relation
  $\ll$ is the following (cf~\cite[Remark~1.5]{MR1164641}): for every
  $u$, $v\in M(\Sigma,\mu)$, one has that
  \begin{center}
    $u\ll v$ if and only if $u^+\ll v^+$ and $u^-\ll v^-$.
  \end{center}
\end{remark}

Further, the relation $\ll$ on $M(\Sigma,\mu)$ has the following
properties. We omit the easy proof of this proposition.

\begin{proposition}\label{propo:properties-of-ll}
  For every $u$, $v$, $w\in M(\Sigma,\mu)$, one has that
  \begin{enumerate}[(1)]
    \item\label{propo:properties-of-ll-claim0} $u^+ \ll u$, $u^{-}\ll -u$;
   \item\label{propo:properties-of-ll-claim1} $u\ll v$ if and only if $u^+\ll v^+$ and $u^-\ll v^-$;
   \item\label{propo:properties-of-ll-claim2} (\emph{positive homogeneity}) if $u\ll v$ then
    $\alpha u\ll \alpha v$ for all $\alpha>0$;
   \item\label{propo:properties-of-ll-claim3} (\emph{transitivity}) if $u\ll v$ and $v\ll w$ then $u\ll w$;
    \item\label{propo:properties-of-ll-claim4} if $u\ll v$ then
      $\abs{u}\ll \abs{v}$;
    \item\label{propo:properties-of-ll-claim5} (\emph{convexity}) for every $u\in M(\Sigma,\mu)$, the set
      $\{w\;\vert\,w\ll u \}$ is convex.
  \end{enumerate}
\end{proposition}

With these preliminaries in mind, we can now state the following
definitions.

\begin{definition}\label{def:complete-contraction}
  A mapping $S : D(S)\to M(\Sigma,\mu)$ with domain $D(S)\subseteq
  M(\Sigma,\mu)$ is called a \emph{complete contraction} if
  \begin{displaymath}
    Su-S\hat{u}\ll u-\hat{u}\qquad\text{for every $u$, $\hat{u}\in D(S)$.}
  \end{displaymath}
  More generally, for $L\ge 1$, we call $S$ to be an
  \emph{$L$-complete contraction} if
  \begin{displaymath}
    L^{-1}Su-L^{-1}S\hat{u}\ll u-\hat{u}\qquad\text{for every $u$, $\hat{u}\in D(S)$,}
  \end{displaymath}
  and for some $\omega\in \R$, $S$ is called to be
  \emph{$\omega$-quasi completely contractive} if $S$ is an
  $L$-complete contraction with $L=e^{\omega t}$ for some $t\ge 0$.
\end{definition}

\begin{remark}
  \label{rem:1}
  Note, for every $1\le q<\infty$, $j_{q}(\cdot)=\abs{[\cdot]^{+}}^{q}\in \mathcal{J}_{0}$,
  $j_{\infty}(\cdot)=[[\cdot]^{+}-k]^{+}\in \mathcal{J}_{0}$ for every 
  $k\ge 0$ (and for large enough $k>0$ if $q=\infty$), and for every
  $N$-function $\psi$ and $\alpha>0$, 
  $j_{\psi,\alpha}(\cdot)=\psi(\frac{[\cdot]^{+}}{\alpha})\in \mathcal{J}_{0}$. This
  shows that for every $L$-complete
  contraction $S : D(S)\to M(\Sigma,\mu)$ with domain
  $D(S)\subseteq M(\Sigma,\mu)$, the mapping $L^{-1}S$ is order-preserving and
  contractive respectively for every $L^{q}$-norm ($1\le q\le \infty$), and
  every $L^{\psi}$-norm with $N$-function $\psi$.
\end{remark}

Now, we can state the
definition of completely accretive operators.

\begin{definition}
  \label{def:completely-accretive-operators}
  An operator $A$ on $M(\Sigma,\mu)$ is called \emph{completely
    accretive} if for every $\lambda>0$, the resolvent operator
  $J_{\lambda}$ of $A$ is a complete contraction, or equivalently, if
  for every $(u_1,v_1)$, $(u_{2},v_{2}) \in A$ and $\lambda >0$, one
  has that
  \begin{displaymath}
    u_1 - u_2 \ll u_1 - u_2 + \lambda (v_1 - v_2).
  \end{displaymath}
  If $X$ is a linear subspace of $M(\Sigma,\mu)$ and $A$ an operator
  on $X$, then $A$ is said to be \emph{$m$-completely accretive on $X$} if $A$ is
  completely accretive and satisfies the \emph{range
  condition}~\eqref{eq:11}.
Further, for $\omega\in \R$, an operator $A$ on a linear subspace
$X\subseteq M(\Sigma,\mu)$ is called \emph{$\omega$-quasi
  ($m$)-completely accretive} in $X$ if $A+\omega I$ is
($m$)-completely accretive in $X$. Finally, an operator $A$ on a
linear subspace $X\subseteq M(\Sigma,\mu)$ is called \emph{quasi
 ($m$-)completely accretive} if there is an $\omega\in \R$ such that
$A+\omega I$ is ($m$-)completely accretive in $X$.
\end{definition}

\begin{remark}
  For $\omega\in \R$, the fact that $A$ is $\omega$-quasi
  ($m$)-completely accretive in $X$ implies that
  the resolvent operator $J_{\lambda}^{A}$ of $A$ is
  $L$-completely contractive for $L=(1-\lambda\omega)^{-1}$ for every
  $\lambda>0$ satisfying $\lambda \omega<1$. Indeed, if $A$ is $\omega$-quasi
  ($m$)-completely accretive in $X$ then by taking 
  $L=(1-\lambda\omega)^{-1}$, one sees that
  \begin{align*}
    \int_{\Sigma}j\Big(L^{-1}(u_{1}-u_{2})\Big)\,\td\mu
    &= \int_{\Sigma}j_{L^{-1}}(u_{1}-u_{2})\,\td\mu\\
    &\le \int_{\Sigma}j_{L^{-1}}\Big(u_{1}-u_{2}
      +\tfrac{\lambda}{1-\lambda\omega}(\omega(u_{1}-u_{2})+v_{1}-v_{2})\Big)\,\td\mu\\
    &=\int_{\Sigma}j(u_{1}-u_{2}+\lambda(v_{1}-v_{2}))\,\td\mu
  \end{align*}
for every $(u_{1},v_{1})$, $(u_{1},v_{1})\in A$ and $\lambda>0$
satisfying $\lambda \omega<1$, where we used that
$j_{L^{-1}}(s):=j((1-\lambda w)s)$ belongs to $J_{0}$.
\end{remark}

This property transfers as follows to the semigroup $\{T_{t}\}_{t\ge 0}$

Before stating a useful characterization of quasi completely accretive
operators, we first need to introduce the following function
spaces. Let
\begin{displaymath}
  L^{1+\infty}(\Sigma,\mu):= L^1(\Sigma,\mu) +
  L^{\infty}(\Sigma,\mu)\;\text{ and }\;
  L^{1\cap \infty}(\Sigma,\mu):=L^{1}(\Sigma,\mu)\cap L^{\infty}(\Sigma,\mu)
\end{displaymath}
be the \emph{sum} and the \emph{intersection space} of
$L^1(\Sigma,\mu)$ and $L^{\infty}(\Sigma,\mu)$, which are
equip\-ped, respectively, with the norms
\begin{align*}
  \norm{u}_{1+\infty}&:= \inf \left\{ \Vert u_1 \Vert_1 + \Vert
  u_2 \Vert_{\infty} \Big\vert u = u_1 + u_2, \ u_1 \in L^1(\Sigma,\mu), u_2
  \in L^{\infty}(\Sigma,\mu) \right\},\\
    \norm{u}_{1 \cap \infty}&:= \max\big\{ \norm{u}_{1}, \norm{u}_{\infty}\big\}
\end{align*}
are Banach spaces. In fact, $L^{1+\infty}(\Sigma,\mu)$ and and
$L^{1\cap \infty}(\Sigma,\mu)$ are respectively the largest and the
smallest of the rearrangement-invariant Banach function spaces
(cf.,~\cite[Chapter~3.1]{MR928802}). If $\mu(\Sigma)$ is finite, then
$L^{1+\infty}(\Sigma,\mu)=L^{1}(\Sigma,\mu)$ with equivalent norms, but
if $\mu(\Sigma)=\infty$ then $L^{1+\infty}(\Sigma,\mu)$ contains
$\bigcup_{1\le q\le\infty}L^{q}(\Sigma,\mu)$. Further, we will employ
the space
\begin{displaymath}
  L_0(\Sigma,\mu):= \left\{ u \in M(\Sigma,\mu) \;\Big\vert\,\Big.
    \int_{\Sigma} \big[\abs{u}
   - k\big]^+\,\td\mu < \infty \text{ for all $k > 0$} \right\},
\end{displaymath}
which equipped with the $L^{1+\infty}$-norm is a closed subspace of
$L^{1+\infty}(\Sigma,\mu)$. In fact, one has (cf.,~\cite{MR1164641})
that
\begin{math}
  L_0(\Sigma,\mu) = \overline{L^1(\Sigma,\mu) \cap
  L^{\infty}(\Sigma,\mu)}^{\mbox{}_{1+\infty}}.
\end{math}
Since for every $k\ge 0$,
$T_{k}(s):=[\abs{s}-k]^+$ is a Lipschitz mapping $T_{k} : \R\to\R$ and
by Chebyshev's inequality, one sees that
$L^{q}(\Sigma,\mu)\hookrightarrow L_0(\Sigma,\mu)$ for every
$1\le q<\infty$ (and $q=\infty$ if $\mu(\Sigma)<+\infty$), and
$L^{\psi}(\Sigma,\mu)\hookrightarrow L_0(\Sigma,\mu)$ for every
$N$-function $\psi$.

\begin{proposition}[{\cite{CoulHau2016}}]
  \label{prop:completely-accretive}
   Let $P_{0}$ denote the set of all functions $T\in C^{\infty}(\R)$
   satisfying $0\le T'\le 1$ such that $T'$  is compactly supported,
   and $x=0$ is not contained in the
   support $\textrm{supp}(T)$ of $T$. Then for $\omega\in \R$,
   an operator $A \subseteq L_{0}(\Sigma,\mu)\times
  L_{0}(\Sigma,\mu)$ is $\omega$-quasi completely accretive if and only if
   \begin{displaymath}
     \int_{\Sigma}T(u-\hat{u})(v-\hat{v})\,\dmu
     +\omega \int_{\Sigma}T(u-\hat{u})(u-\hat{u})\,\dmu\ge 0
   \end{displaymath}
   for every $T\in P_{0}$ and every $(u,v)$, $(\hat{u},\hat{v})\in A$.
 \end{proposition}

\begin{remark}
  For convenience, we denote the unique extension of
  $\{T_{t}\}_{t\ge 0}$ on $L^\psi(\Sigma,\mu)$ or $L^1(\Sigma,\mu)$
  again by $\{T_{t}\}_{t\ge 0}$.
\end{remark}

\begin{definition}\label{def:normal-Banach-space}
  A Banach space $X\subseteq M(\Sigma,\mu)$ with norm
  $\norm{\cdot}_{X}$ is called \emph{normal} if the norm
  $\norm{\cdot}_{X}$ has the following property:
  \begin{displaymath}
    \begin{cases}
      & \text{for every $u\in X$, $v\in M(\Sigma,\mu)$ satisfying}\quad
        v\ll u,\\
      & \text{one has that}\quad v\in X\quad\text{ and }\quad\norm{v}_{X}\le
      \norm{u}_{X}.
    \end{cases}
  \end{displaymath}
\end{definition}

Typical examples of normal Banach spaces $X\subseteq M(\Sigma,\mu)$
are Orlicz-spaces $L^{\psi}(\Sigma,\mu)$ for every $N$-function
$\psi$, $L^{q}(\Sigma,\mu)$, ($1\le q\le \infty$),
$L^{1\cap \infty}(\Sigma,\mu)$, $L_0(\Sigma,\mu)$, and
$L^{1+\infty}(\Sigma,\mu)$.

\begin{remark}
  \label{rem:2}
  It is important to point out that if $X$ is a normal Banach space,
  then for every $u\in X$, one always has that $u^+$, $u^-$ and
  $\abs{u}\in X$. To see this, recall that
  by~\eqref{propo:properties-of-ll-claim0}
  Proposition~\ref{propo:properties-of-ll}, if $u\in X$, then $u^+\ll
  u$ and $u^-\ll -u$. Thus, $u^+$ and $u^-\in X$ and since
  $\abs{u}=u^+ +u^-$, one also has that $\abs{u}\in X$.
\end{remark}

The dual space
$(L_0(\Sigma,\mu))^{\mbox{}_{\prime}}$ of $L_0(\Sigma,\mu)$ is isometrically
isomorphic to the space $L^{1\cap \infty}(\Sigma,\mu)$. Thus, a sequence
$(u_{n})_{n\ge 1}$ in $L_0(\Sigma,\mu)$ is said to be \emph{weakly
  convergent} to $u$ in $L_0(\Sigma,\mu)$ if
\begin{displaymath}
  \langle v,u_{n}\rangle:=\int_{\Sigma} v\,u_{n}\,\td\mu\to
  \int_{\Sigma} v\,u\,\td\mu\qquad
  \text{for every $v\in L^{1\cap \infty}(\Sigma,\mu)$.}
\end{displaymath}
For the rest of this paper, we write $\sigma(L_{0},L^{1\cap\infty})$ to denote
the \emph{weak topology} on $L_0(\Sigma,\mu)$. For this weak topology,
we have the following compactness result.

\begin{proposition}[{\cite[Proposition~2.11]{MR1164641}}]
  \label{propo:compactness-in-L0}
  Let $u\in L_0(\Sigma,\mu)$. Then, the following statements hold.
  \begin{enumerate}[(1)]
   \item\label{propo:compactness-in-L0-claim1} The set
    \begin{math}
      \Big\{v\in M(\Sigma,\mu)\;\Big\vert\;v\ll u\Big\}
    \end{math}
    is $\sigma(L_{0},L^{1\cap\infty})$-sequentially compact in
    $L_0(\Sigma,\mu)$;
  \item\label{propo:compactness-in-L0-claim2} Let
    $X\subseteq M(\Sigma,\mu)$ be a normal Banach space satisfying
    $X\subseteq L_0(\Sigma,\mu)$ and
     \begin{equation}
       \label{eq:7}
       \begin{cases}
         &\text{for every $u\in X$,
           $(u_{n})_{n\ge 1}\subseteq M(\Sigma,\mu)$ with $u_{n}\ll u$
           for all $n\ge 1$}\\[7pt]
         &\text{and $\displaystyle\lim_{n\to+\infty}u_{n}(x)=u(x)$ $\mu$-a.e. on
           $\Sigma$, yields}
         \quad \displaystyle\lim_{n\to +\infty}u_{n}=u\text{ in $X$.}
       \end{cases}
     \end{equation}
     Then for every $u\in X$ and sequence
     $(u_{n})_{n\ge 1}\subseteq M(\Sigma,\mu)$ satisfying
     \begin{displaymath}
       u_{n}\ll u\text{ for
     all $n\ge 1$}\quad\text{and}\quad \lim_{n\to+\infty}u_{n}=u
     \text{ $\sigma(L_{0},L^{1\cap\infty})$-weakly in $X$,}
   \end{displaymath}
   one has that
     \begin{displaymath}
       \lim_{n\to +\infty}u_{n}=u\qquad\text{ in $X$.}
     \end{displaymath}
  \end{enumerate}
\end{proposition}

Note, examples of normal Banach spaces $X\subseteq L_0(\Sigma,\mu)$
satisfying~\eqref{eq:7} are $X=L^{p}(\Sigma,\mu)$ for $1\le p<\infty$
and $L_{0}(\Sigma,\mu)$.\medskip

To complete this preliminary section, we state the following Proposition
summarizing statements from~\cite{CoulHau2016}, which we will need in the sequel
(cf.,~\cite{MR1164641} for the case $\omega=0$).

\begin{proposition}\label{propo:properties-of-St}
  For $\omega\in \R$, let $A$ be $\omega$-quasi completely accretive in
  $L_{0}(\Sigma,\mu)$.
  \begin{enumerate}[(1.)]
  \item\label{propo:properties-of-St-claim1} If there is a $\lambda_{0}>0$ such that
    $\textrm{Rg}(I+\lambda A)$ is dense in $L_{0}(\Sigma,\mu)$, then
    for the closure $\overline{A}^{\mbox{}_{L_{0}}}$ of $A$ in $L_{0}(\Sigma,\mu)$ and
    every normal Banach space with $X\subseteq L_0(\Sigma,\mu)$, the
    restriction $\overline{A}_{X}^{\mbox{}_{L_{0}}}:=\overline{A}^{\mbox{}_{L_{0}}}\cap(X\times X)$ of
    $A$ on $X$ is the unique $\omega$-quasi $m$-completely accretive extension of the
    part $A_{X}=A \cap(X\times X)$ of $A$ in $X$.

  \item For a given normal Banach space $X\subseteq L_0(\Sigma,\mu)$,
    and $\omega\in \R$, suppose $A$ is $\omega$-quasi $m$-completely accretive in $X$, and
    $\{T_{t}\}_{t\ge 0}$ be the semigroup generated by $-A$ on
    $\overline{D(A)}^{\mbox{}_{X}}$. Further, let $\{S_{t}\}_{t\ge 0}$
    be the semigroup generated by $-\overline{A}^{\mbox{}_{L_{0}}}$, where
    $\overline{A}^{\mbox{}_{L_{0}}}$ denotes the closure of $A$ in
    $\overline{X}^{\mbox{}_{L_{0}}}$. Then, the following
    statements hold.
    \begin{enumerate}[(a)]
    \item\label{propo:properties-of-St-claim1} The semigroup
      $\{S_{t}\}_{t\ge 0}$ is $\omega$-quasi completely contractive on
      $\overline{D(A)}^{\mbox{}_{L_{0}}}$, $T_{t}$ is the
      restriction of $S_{t}$ on $\overline{D(A)}^{\mbox{}_{X}}$,
      $S_{t}$ is the closure of $T_{t}$ in $L_{0}(\Sigma,\mu)$,
      and
      \begin{equation}
        \label{eq:exp-form-St}
        S_{t}u_{0}= L_{0}-\lim_{n\to+\infty}
        \left(I+\frac{t}{n}A\right)^{-n}u_{0}\quad\text{for all $u_{0}\in \overline{D(A)}^{\mbox{}_{L_{0}}}\cap X$;}
      \end{equation}
    \item\label{propo:properties-of-St-claim5} If there exists
      $u\in L^{1\cap \infty}(\Sigma,\mu)$ such that the orbit
      $\{T_{t}u\,\vert\,t\ge 0\}$ is locally bounded on $\R_+$ with
      values in $L^{1\cap\,\infty}(\Sigma,\mu)$, then, for every
      $N$-function $\psi$, the semigroup $\{T_{t}\}_{t\ge 0}$ can be
      extrapolated to a strongly continuous, order-preserving
      semigroup of $\omega$-quasi contractions on
      $\overline{\overline{D(A)}^{\mbox{}_{X}}\cap
        L^{1\cap\,\infty}(\Sigma,\mu)}^{\mbox{}_{L^\psi}}$
      (respectively, on $\overline{\overline{D(A)}^{\mbox{}_{X}}\cap
        L^{1\cap\,\infty}(\Sigma,\mu)}^{\mbox{}_{L^1}}$), and to an
      order-preserving semigroup of $\omega$-quasi contractions on
      $\overline{\overline{D(A)}^{\mbox{}_{X}}\cap
        L^{1\cap\,\infty}(\Sigma,\mu)}^{\mbox{}_{L^\infty}}$. We
      denote each extension of $T_{t}$ on on those spaces again by
      $T_{t}$.

    \item\label{propo:properties-of-St-claim2} The restriction
      $A_{X}:=\overline{A}^{\mbox{}_{L_{0}}}\cap (X\times X)$ of $\overline{A}^{\mbox{}_{L_{0}}}$ on $X$
      is the unique $\omega$-quasi $m$-complete extension of $A$ in
      $X$; that is, $A=A_{X}$.
    \item\label{propo:properties-of-St-claim3} The operator $A$ is
      sequentially closed in $X\times X$ equipped with the
      relative\newline
      $(L_0(\Sigma,\mu)\times
      (X,\sigma(L_{0},L^{1\cap\infty})))$-topology.
    \item\label{propo:properties-of-St-claim3bis} The domain of $A$ is
      characterized by
      \begin{displaymath}
        \mbox{}\qquad D(A)=\Bigg\{u\in \overline{D(A)}^{\mbox{}_{L_{0}}}\cap
        X\,\Bigg\vert\;
        \begin{array}[c]{l}
          \exists\;v\in X\text{ such that }\\
          e^{-\omega t}\frac{S_{t}u-u}{t}\ll v\text{ for all small $t>0$}
        \end{array}
        \Bigg\};
      \end{displaymath}
    \item\label{propo:properties-of-St-claim4} For every $u\in D(A)$,
      one has that
      \begin{equation}
        \label{eq:limit-infinit-gen-in-omega-A}
        \lim_{t\to 0+}\frac{S_{t}u-u}{t}=-A^{\!
          \circ}u\qquad\text{strongly in $L_{0}(\Sigma,\mu)$.}
      \end{equation}
    \end{enumerate}
  \end{enumerate}
\end{proposition}

%
%

\subsection{Regularizing effect of the associated semigroup}
\label{subsec:homogeneous-zero}
It is worth recalling that the Banach space
$L^{1}(\Sigma,\mu)$ does not admit the Radon-Niko\-d\'ym
property. Thus, the time-derivative
$\tfrac{\td }{\dt_{+}}T_{t}u_{0} (t)$, $u_{0}\in L^{1}(\Sigma,\mu)$, of a given semigroup
$\{T_{t}\}_{t\ge 0}$ on $L^{1}(\Sigma,\mu)$ does not need to exist in
$L^{1}(\Sigma,\mu)$. But, in this section, we show that even though the
underlying Banach space $X$ is not reflexive, if the infinitesimal generator $-A$ is homogeneous 
of order $\alpha\neq 1$ and $A$ is quasi-completely accretive, then 
the time-derivative $\tfrac{\td u}{\dt_{+}}(t)$ exists in $X$. This fact follows from the
following compactness result generalizing the one in~\cite{MR1164641}
for $\omega=0$. 

Here, the partial ordering
``$\!\le$'' is the standard one defined by $u\le v$ for $u$,
$v\in M(\Sigma,\mu)$ if $u(x)\le v(x)$ for $\mu$-a.e.  $x\in \Sigma$,
and we write$X\hookrightarrow Y$ for indicating that the space
$X$ is continuously embedded into the space $Y$.

\begin{lemma}\label{lem:compacteness-of-complete-semigroup}
  Let $X\subseteq L_{0}(\Sigma,\mu)$ be a normal Banach space
  satisfying~\eqref{eq:7}. For $\omega\in \R$, let
  $\{T_{t}\}_{t\ge 0}$ be a family of mappings $T_{t} : C\to C$
  defined on a subset $C\subseteq X$ of $\omega$-quasi complete
  contractions satisfying~\eqref{eq:21} and $T_{t}0=0$ for all
  $t\ge 0$. Then, for every $u_{0}\in C$ and $t>0$, the set
  \begin{equation}
    \label{eq:53}
    \Bigg\{\frac{T_{t+h}u_{0}-T_{t}u_{0}}{h}\,\Bigg\vert\Bigg.\,h\neq
    0, t+h>0\Bigg\}
  \end{equation}
  is $\sigma(L_{0},L^{1\cap\infty})$-weakly sequentially compact in
  $L_{0}(\Sigma,\mu)$.
\end{lemma}

The proof of this lemma is essentially the same as in the case
$\omega=0$ (cf.,~\cite{MR1164641}). For the convenience of the reader,
we include here the proof.

\begin{proof}
  Let $u_{0}\in C$, $t>0$, and $h\neq 0$ such that $t+h>0$. Then by taking
  $\lambda=1+\frac{h}{t}$ in~\eqref{eq:21}, one sees that
  \begin{align*}
    \abs{T_{t+h}u_{0}-T_{t}u_{0}}&= \abs{\lambda^{\frac{1}{1-\alpha}}
                             \,T_{t}\left[\lambda^{\frac{1}{\alpha-1}}u_{0}\right]-T_{t}u_{0}}\\
    & \le
      \lambda^{\frac{1}{1-\alpha}}\,\labs{T_{t}\left[\lambda^{\frac{1}{\alpha-1}}u_{0}\right]-T_{t}u_{0}}
+\abs{\lambda^{\frac{1}{1-\alpha}}-1}\,
    \abs{T_{t}u_{0}}.
  \end{align*}
  Since $T_{t}$ is an $\omega$-quasi complete contraction and since $T_{t}0=0$, ($t\ge 0$),
  claim~\eqref{propo:properties-of-ll-claim2} and \eqref{propo:properties-of-ll-claim4} of
  Proposition~\ref{propo:properties-of-ll} imply that
  \begin{displaymath}
    \lambda^{\frac{1}{1-\alpha}}\,
    e^{-\omega t}\labs{T_{t}\left[\lambda^{\frac{1}{\alpha-1}}u_{0}\right]-T_{t}u_{0}}\ll
    \abs{1-\lambda^{\frac{1}{1-\alpha}}}\,\abs{u_{0}}
  \end{displaymath}
  and
  \begin{displaymath}
    \abs{\lambda^{\frac{1}{1-\alpha}}-1}\, e^{-\omega t}
    \abs{T_{t}u_{0}}\ll \abs{\lambda^{\frac{1}{1-\alpha}}-1}\,
    \abs{u_{0}}.
  \end{displaymath}
  Since the set $\{w\,\vert\,w\ll \abs{\lambda^{\frac{1}{1-\alpha}}-1}\,
  \abs{u_{0}}\}$ is convex (cf.,~\eqref{propo:properties-of-ll-claim5} of
  Proposition~\ref{propo:properties-of-ll}), the previous inequalities
  imply that
  \begin{displaymath}
   \frac{1}{2} e^{-\omega t}\abs{T_{t+h}u_{0}-T_{t}u_{0}}\ll\, \abs{\lambda^{\frac{1}{1-\alpha}}-1}\,
    \abs{u_{0}}.
  \end{displaymath}
  Using again~\eqref{propo:properties-of-ll-claim2} of
  Proposition~\ref{propo:properties-of-ll}, gives
  \begin{equation}
    \label{eq:5}
    \frac{\abs{T_{t+h}u_{0}-T_{t}u_{0}}}{\abs{\lambda^{\frac{1}{1-\alpha}}-1}}\ll\,2\, e^{\omega t}\,\abs{u_{0}}.
  \end{equation}
  Since for every $u\in M(\Sigma,\mu)$, one always has that
  $u^+\ll\abs{u}$, the transitivity of ``$\!\ll$'' (\eqref{propo:properties-of-ll-claim3} of
  Proposition~\ref{propo:properties-of-ll}) implies that
  \begin{displaymath}
   f_{h}:=\frac{T_{t+h}u_{0}-T_{t}u_{0}}{\lambda^{\frac{1}{1-\alpha}}-1}\quad\text{
    satisfies }\quad
   f_{h}^{+}\ll\,2\, e^{\omega t}\, 
   \abs{u_{0}}.
  \end{displaymath}
  Therefore and since $\abs{u_{0}}\in X$, \eqref{propo:compactness-in-L0-claim1} of
  Proposition~\ref{propo:compactness-in-L0} yields that the two sets
  $\{f_{h}^{+}\vert\,h\neq 0, t+h>0\}$ and
  $\{\abs{f_{h}}\vert\,h\neq 0, t+h>0\}$ are
  $\sigma(L_{0},L^{1\cap\infty})$- weakly sequentially compact in
  $L_{0}(\Sigma,\mu)$. Since $f_{h}^{-}=\abs{f_{h}}-f_{h}^{+}$ and
  $f_{h}=f_{h}^{+}-f_{h}^{-}$, and since
  $(\lambda^{\frac{1}{1-\alpha}}-1)/h=((1+\frac{h}{t})^{\frac{1}{1-\alpha}}-1)/h\to
  1/t(1-\alpha)\neq 0$ as $h\to 0$, we can conclude that the claim of
  this lemma holds.
\end{proof}

With these preliminaries in mind, we can now state the regularization
effect of the semigroup $\{T_{t}\}_{t\ge 0}$ generated by a
$\omega$-quasi $m$-completely accretive operator of homogeneous order
$\alpha\neq 1$.

\begin{theorem}\label{thm:regularity-of-complete-acc-homogen-operators}
  Let $X\subseteq L_0(\Sigma,\mu)$ be a normal Banach space
  satisfying~\eqref{eq:7}, and $\norm{\cdot}$ denote the norm on
  $X$. For $\omega\in \R$, let $A$ be $\omega$-quasi $m$-completely
  accretive in $X$, and $\{T_{t}\}_{t\ge 0}$ be the semigroup
  generated by $-A$ on $\overline{D(A)}^{\mbox{}_{X}}$. If $A$ is
  homogeneous of order $\alpha\neq 1$, then for every
  $u_{0}\in \overline{D(A)}^{\mbox{}_{X}}$ and $t>0$,
  $\frac{\td T_{t}u_{0}}{\dt}$ exists in $X$ and
  \begin{equation}
    \label{eq:3}
    \abs{A^{\! \circ}T_{t}u_{0}}\le \frac{2 e^{\omega
      t}}{\abs{\alpha-1}}\,\frac{\abs{u_{0}}}{t}\qquad\text{$\mu$-a.e. on $\Sigma$.}
  \end{equation}
  In particular, for every $u_{0}\in \overline{D(A)}^{\mbox{}_{X}}$,
  \begin{equation}
    \label{eq:6}
    \lnorm{\frac{\td T_{t}u_{0}}{\dt}_{\!\!+}}\le \frac{2 e^{\omega
      t}}{\abs{\alpha-1}} \frac{\norm{u_{0}}}{t}
   \qquad\text{for every $t>0$,}
  \end{equation}
  and 
  \begin{equation}
    \label{eq:6bis}
    \frac{\td T_{t}u_{0}}{\dt}_{\!\!+}\ll \frac{2 e^{\omega t}}{\abs{\alpha-1}} \frac{\abs{u_{0}}}{t}
   \qquad\text{for every $t>0$.}
  \end{equation}
\end{theorem}

\begin{proof}
  Let $u_{0}\in \overline{D(A)}^{\mbox{}_{X}}$, $t>0$, and
  $(h_{n})_{n\ge 1}\subseteq \R$ be a zero sequence such that
  $t+h_{n}>0$ for all $n\ge 1$. Then, by Proposition~\ref{acrethomog},
  we can apply the compactness result stated in
  Lemma~\ref{lem:compacteness-of-complete-semigroup}. Thus, there is a
  $z\in L_{0}(\Sigma,\mu)$ and a subsequence $(h_{k_{n}})_{n\ge 1}$ of
  $(h_{n})_{n\ge 1}$ such that
  \begin{equation}
    \label{eq:1}
    \lim_{n\to
      \infty}\frac{T_{t+h_{k_{n}}}u_{0}-T_{t}u_{0}}{h_{k_{n}}}=z\qquad\text{weakly
      in $L_{0}(\Sigma,\mu)$.}
  \end{equation}
  By~\eqref{propo:properties-of-St-claim3bis} of
  Proposition~\ref{propo:properties-of-St}, 
  one has that $(T_{t}u_{0},-z)\in A$. Thus
  \eqref{propo:properties-of-St-claim4} of
  Proposition~\ref{propo:properties-of-St} yields that
  $z=-A^{\!\circ}T_{t}u_{0}$ and 
  \begin{equation}
    \label{eq:2-new}
    \lim_{n\to \infty}\frac{T_{t+h_{k_{n}}}u_{0}-T_{t}u_{0}}{h_{k_{n}}}=-A^{\!\circ}T_{t}u_{0}
    \qquad\text{strongly in $L_{0}(\Sigma,\mu)$.}
  \end{equation}
  After possibly passing to another subsequence, the
  limit~\eqref{eq:2-new} also holds $\mu$-a.e. on $\Sigma$. The
  argument shows that the limit~\eqref{eq:2-new} is independent of the
  choice of the initial zero sequence $(h_{n})_{n\ge 1}$. Thus
  \begin{equation}
    \label{eq:56}
    \lim_{h\to 0}\frac{T_{t+h}u_{0}-T_{t}u_{0}}{h}=-A^{\!
      \circ}T_{t}u_{0}\qquad\text{exists $\mu$-a.e. on $\Sigma$.}
  \end{equation}
  Since $2 e^{-\omega t}\,\abs{u_{0}}\in X$, by~\eqref{eq:5}, and since
  $(\lambda^{\frac{1}{1-\alpha}}-1)/h=((1+\frac{h}{t})^{\frac{1}{1-\alpha}}-1)/h\to
  1/t(1-\alpha)\neq 0$ as $h\to 0$, it follows from \eqref{propo:compactness-in-L0-claim2} of
  Proposition~\ref{propo:compactness-in-L0} that 
 \begin{equation}
    \label{eq:8}
    \lim_{h\to 0}\frac{T_{t+h}u_{0}-T_{t}u_{0}}{h}=-A^{\!
      \circ}T_{t}u_{0}\qquad\text{exists in $X$}
  \end{equation}
   and with $\lambda=1+\tfrac{h}{t}$,
  \begin{displaymath}
    \frac{\abs{T_{t+h}u_{0}-T_{t}u_{0}}}{\abs{\lambda^{\frac{1}{1-\alpha}}-1}}\le\,2\,
    e^{-\omega t}\,\abs{u_{0}}
  \end{displaymath}
  for all $h\neq 0$ satisfying $t+h>0$. Sending $h\to 0$ in the last
  inequality and applying~\eqref{eq:8} gives~\eqref{eq:3}. In
  particular, by Corollary~\ref{thm:1}, one has that~\eqref{eq:6}
  holds for the norm $\norm{\cdot }_{X}$ on
  $X$. Moreover,~\eqref{eq:5} is equivalent to 
  \begin{equation}
    \label{eq:57}
    \int_{\Sigma} j\left(\frac{\abs{T_{t+h}u_{0}-T_{t}u_{0}}}{\abs{\lambda^{\frac{1}{1-\alpha}}-1}}\right)
  \,\td\mu \le \int_{\Sigma} j\left(2\,
    e^{-\omega t}\,\abs{u_{0}}\right) \, \td\mu
  \end{equation}
 for all $h\neq 0$ satisfying $t+h>0$, and every $j\in J_0$. By the
 lower semicontinuity of $j\in J_{0}$ and by the
 $\mu$-a.e. limit~\eqref{eq:56}, we have that
 \begin{displaymath}
   j\left(\frac{\td T_{t}u_{0}}{\dt}(x)\,\abs{\alpha-1}\,t\right)
   \le \liminf_{h\to 0}j\left(\frac{\abs{T_{t+h}u_{0}(x)-T_{t}u_{0}(x)}}{\abs{\lambda^{\frac{1}{1-\alpha}}-1}}\right)
 \end{displaymath}
 for $\mu$-a.e. $x\in \Sigma$. Thus, taking the limit inferior as $h\to 0+$
 in~\eqref{eq:57} and applying Fatou's lemma yields
 \begin{displaymath}
   \int_{\Sigma}  j\left( \frac{\td T_{t}u_{0}}{\dt}(x)\,\abs{\alpha-1}\,t\right)
  \,\td\mu \le \int_{\Sigma} j\left(2\,
    e^{-\omega t}\,\abs{u_{0}}\right) \, \td\mu
 \end{displaymath}
Since $j\in J_{0}$ was arbitrary and by \eqref{propo:properties-of-ll-claim2} of
Proposition~\ref{propo:properties-of-ll}, this shows that~\eqref{eq:6bis}
holds and thereby completes the proof of this theorem.
\end{proof}

\section{Application}
\label{sec:application}

\subsection{An elliptic-parabolic boundary-value problem}

Our aim in this section is to derive global $L^1$
Aronson-B\'enilan estimate~\eqref{eq:29} for $X=L^{q}(\partial M)$,
($1\le q\le \infty$), and point-wise Aronson-B\'enilan
estimate~\eqref{eq:14} on the time-derivative $\frac{\td u}{\dt}$ of
any solutions $u$ to the
elliptic-parabolic boundary-value problem
\begin{equation}
  \label{eq:12}
  \begin{cases}
    \hspace{1.95cm}
    -\Delta_{p}u+m\,\abs{u}^{p-2}u=0&\qquad\text{in $M\times (0,\infty)$,}\\
    \partial_{t}u+\abs{\nabla u}^{p-2}_{g}\nabla u\cdot\nu + f(x,u)=0
    &\qquad\text{on $\partial M\times
      (0,\infty)$,}\\
    \hspace{4.29cm}u(0)=u_{0}&\qquad\text{on $\partial M$.}\\
  \end{cases}
\end{equation}
Here, we assume that $1<p<\infty$, 
$\Delta_{p}$ denotes the celebrated \emph{$p$-Laplace-Beltrami
  operator}
\begin{equation}
  \label{eq:35}
 \Delta_{p}u:=\divi\left(\abs{\nabla u}^{p-2}_{g}\nabla
   u\right)\qquad\text{in $\mathcal{D}'(M)$}
\end{equation}
for $u\in W^{1,p}(M)$ on a compact, smooth, $N$-dimensional Riemannian
manifold $(M,g)$ with a Lipschitz continuous boundary $\partial M$,
$m>0$ and
$f : \partial M\times \R\to \R$ a Lipschitz-continuous
Carath\'eodory function
(see~\eqref{property:one-Charatheodory}-\eqref{eq:2} below). 

For applying the theory developed in the previous sections of this
paper, it is worth noting that the elliptic-parabolic
problem~\eqref{eq:12} can be rewritten in the form of the perturbed
Cauchy problem~\eqref{eq:10bis-no-F} in the Banach space
$X=L^{q}(\partial M)$, ($1\le q\le \infty$), where the operator $A$ is
the \emph{Dirichlet-to-Neumann operator} realized in $X$ associated
with the operator $-\Delta_{p}+m\,\abs{\cdot}^{p-2}\cdot$; that is,
$A$ assigns Dirichlet data $\varphi$ on $\partial M$ to the
\emph{co-normal derivative} $\abs{\nabla u}^{p-2}_{g}\nabla u\cdot\nu$
on $\partial M$, where $u$ is the unique weak solution of the
\emph{Dirichlet problem}
  \begin{equation}
    \label{eq:19}
  \begin{cases}
    -\Delta_{p}u+m\,\abs{u}^{p-2}u=0&\qquad\text{in $M$,}\\
    \phantom{-\Delta_{p}u+m\,\abs{u}^{p-2}}u=\varphi&\qquad\text{on $\partial M$.}
  \end{cases}
\end{equation}
In the (flat) case $M=\Omega$ is a bounded domain in $\R^{N}$ with a
Lipschitz-continuous boundary $\partial\Omega$, the
Dirichlet-to-Neumann operator $A$ associated with the
$p$-Laplace-Beltrami operator $\Delta_{p}$ and its semigroup
$\{T_{t}\}_{t\ge 0}$ were studied in the past by several authors (see, for
instance, in~\cite{MR3369257,MR3465809,CoulHau2016} and the references
therein).

\subsection{Framework}

Throughout this section, let $(M,g)$ denote a compact, smooth, (orientable),
$N$-dim\-ensional Riemannian manifold with a Lipschitz continuous
boundary $\partial M$. Let $g= \{g(x)\}_{x\in M}$ denote the corresponding
Riemannian metric tensor and for every $x\in M$, $T_{x}$ be the tangent space
and $TM$ the tangent bundle of $M$. We write
$\abs{\xi}_{g}=\sqrt{\langle\xi,\xi\rangle_{g(x)}}$, ($\xi\in T_{x}$),
to denote the induced norm of the inner product
$\langle\cdot,\cdot\rangle_{g(x)}$ on the tangent space $T_{x}$. If
for given $f\in C^{\infty}(M)$, $df$ is the \emph{differential} at
$x\in M$ and for every chart
$(\Omega,\phi)$, $\bm{g}=(g_{ij})_{i,j=1}^{N}$ is the matrix of the
Riemannian metric $g$ on $\Omega$ with inverse $\bm{g}^{-1}$, then the
corresponding gradient of $f$ at $x$ is given by
$\nabla f(x)=\bm{g}^{-1}(x)df(x)$, and for every $C^{1}$-vector field
$X=(X^{1},\dots,X^{N})$ on $M$, the \emph{divergence}
\begin{displaymath}
  \divi(X):=\tfrac{1}{\sqrt{\textrm{det}(\bm{g})}}
\tfrac{\partial}{\partial x_{i}}\left(\sqrt{\textrm{det}(\bm{g})}
  X^{i}\right).
\end{displaymath}
For given $C^{1}$-curve $\kappa$ with parametrization $\gamma_{\kappa} :
[0,1]\to M$, the length $L(\kappa)$ of $\kappa$ is defined by 
\begin{displaymath}
  L(\kappa)=\int_{0}^{1}\labs{\tfrac{\td \gamma_{\kappa}}{\dt}(t)}_{g(\gamma_{\kappa}(t))}\,\dt.
\end{displaymath}
If we denote by $C^{1}_{x,y}$ the space of all piecewise $C^1$-curves $\kappa$
with starting point $\gamma_{\kappa}(0)=x\in M$ and end point $\gamma_{\kappa}(1)=y\in M$, then
$\td_{g}(x,y):=\inf_{\kappa \in C^{1}_{x,y}}L(\kappa)$ defines a
distance (called Riemannian distance) whose induced topology
$\tau_{g}$ coincides with the original one by $M$. 
There exists a unique Borel measure $\mu_{g}$ defined on the Borel
$\sigma$-Algebra $\sigma(\tau_{g})$ such that on any chart
$(\Omega,\phi)$ of $M$, one has that
$\td\mu_{g}=\sqrt{\textrm{det}(\bm{g})}\,\dx$, where $\dx$ refers to the
Lebesgue measure in $\Omega$. 

For the measure space $(M,\mu_{g})$, and $1\le q\le \infty$, we denote
by $L^{q}(M)=L^{q}(M,\mu_{g})$ (respectively,
$L^{q}_{loc}(M)=L^{q}_{loc}(M,\mu_{g})$) the classical Lebesgue space
of (locally) $q$-integrable functions, and we denote by
$\norm{\cdot}_{q}$ its standard norm on $L^{q}(M)$. Since a vector
field $v$ on $M$ is \emph{measurable} if and only if every component
of $v$ is measurable on all charts $U$ of $M$, one defines similarly
for every $1\le q\le \infty$, the space
$\vec{L}^{q}(M)=\vec{L}^{q}(M,\mu_{g})$ (respectively,
$\vec{L}^{q}_{loc}(M)=\vec{L}^{q}_{loc}(M,\mu_{g})$) of all measurable
vector fields $v$ on $M$ such that $\abs{v}\in L^{q}(M,\mu_{g})$
(respectively, $\abs{v}\in L^{q}_{loc}(M,\mu_{g})$).

The \emph{space of test functions} $\mathcal{D}(M)$ be the set
$C^{\infty}_{c}(M)$ of smooth compactly supported functions equipped
with the following type of convergence: given a sequence
$(\varphi_{n})_{n\ge 1}$ in $C^{\infty}_{c}(M)$ and
$\varphi\in C^{\infty}_{c}(M)$, we say $\varphi_{n}\to\varphi$ in
$\mathcal{D}(M)$ if there is a compact subset $K$ of $M$ such that the
support $\textrm{supp}(\varphi_{n})\subseteq K$ for all $n$, and for every chart
$U$, and all multi-index $\alpha$, one has
$D^{\alpha}\varphi_{n}\to D^{\alpha}\varphi$ uniformly on $U$. Then
the \emph{space of distributions} $\mathcal{D}'(M)$ is the topological
dual space of $\mathcal{D}(M)$. Similarly, one defines the \emph{space
  of test vector fields} $\vec{\mathcal{D}}(M)$ on $M$ and
corresponding dual space $\vec{\mathcal{D}}'(M)$ of
\emph{distributional vector fields}. Given a distribution
$T\in \mathcal{D}'(M)$, the \emph{distributional gradient}
$\nabla T\in \vec{\mathcal{D}}'(M)$ is defined by
\begin{displaymath}
  \langle \nabla T,\psi\rangle_{\vec{\mathcal{D}}'(M), \vec{\mathcal{D}}(M)}
  =-\langle T,\divi{\psi}\rangle_{\mathcal{D}'(M), \mathcal{D}(M)}
  \qquad\text{for every $\psi\in \vec{\mathcal{D}}(M)$.}
\end{displaymath}
For given $u\in L^{1}_{loc}(M)$,
\begin{displaymath}
  \langle
  u,\varphi\rangle_{\mathcal{D}'(M),\mathcal{D}(M)}:=\int_{M}u\,\varphi\td\mu_{g},\qquad
 \varphi\in \mathcal{D}(M),
\end{displaymath}
defines a distribution (called \emph{regular distribution}) on $M$. If
the distributional gradient $\nabla u$ of the distribution $u$ belongs
to $\vec{L}^{1}_{loc}(M)$, then $\nabla u$ is called a \emph{weak
  gradient} of $u$. The \emph{first Sobolev space}
$W^{1,q}(M)=W^{1,q}(M,\mu_{g})$ is the space of all $u\in L^{q}(M)$ such that
for the weak gradient $\nabla u$ of $u$ belongs to
$\vec{L}^{q}(M)$. The space $W^{1,q}(M)$ is a Banach space equipped
with the norm
\begin{equation}
  \label{eq:17}
  \norm{u}_{W^{1,q}(M)}:=\left(\norm{u}_{q}+\norm{\abs{\nabla
      u}_{g}}_{q}\right)^{1/q},\qquad(u\in W^{1,q}(M)),
\end{equation}
and $W^{1,q}(M)$ is reflexive if $1<q<\infty$ (cf.,
\cite[Proposition~2.4]{MR1481970}). Further, we denote by
$W^{1,q}_{0}(M)=W_{0}^{1,q}(M,\mu_{g})$ the closure of $C^{\infty}_{c}(M)$ in
$W^{1,q}(M)$. Since we have assumed that $(M,g)$ is compact, the volume
$\mu_{g}(M)$ is finite. Hence by the compactness result of Rellich-Kondrakov
(see, e.g., \cite[Corollary~3.7]{MR1481970}), we
know that a Poincar\'e inequality on $W^{1,q}_{0}(M)$ is available. Thus,
$\norm{\abs{\nabla\cdot}_{g}}_{q}$ defines an equivalent norm
to~\eqref{eq:17} on $W^{1,q}_{0}(M)$. 

Further, let $\bm{s}_{g}$ denote the surface measure on $\partial M$
induced by the outward pointing unit normal on $\partial M$. Then for
$1\le q< \infty$ and $0<s<1$, let
$W^{s,q}(\partial M):=W^{s,q}(\partial M,\bm{s}_{g})$ be the
Sobolev-Slobode\v{c}ki space given by all measurable functions
$u\in L^{q}(\partial M)=L^{q}(\partial M, \bm{s}_{g})$ with finite Gagliardo semi-norm
\begin{displaymath}
  [u]_{W^{s,q}(\partial M)}^{q}:=\int_{\partial M}\int_{\partial
    M}\tfrac{\abs{u(x)-u(y)}^{q}}{d^{N-2+sq}_{g}(x,y)}
  \td \bm{s}_{g}(x)\td\bm{s}_{g}(y).
\end{displaymath}
The space $W^{s,q}(\partial M)$ equipped with the norm
\begin{displaymath}
  \norm{u}_{W^{s,q}(\partial M)}:=\left(\norm{u}^q_{L^{q}(\partial
    M)}+[u]_{W^{s,q}(\partial M)}^{q}\right)^{1/q}
\end{displaymath}
is a Banach space, which is reflexive if $1<q<\infty$.

Since $M$ is compact, $M$ can be covered by a finite family
$((\Omega_{l},\phi_{l}))_{l=1}^{K}$ of charts $(\Omega_{l},\phi_{l})$
such that for every $l\in \{1,\dots, K\}$, each component $g_{ij}$ of
the matrix $\bm{g}$ of the Riemannian metric $g$ satisfies
\begin{equation}
  \label{eq:49}
  \frac{c_{l}}{2}\,\delta_{ij}\le g_{ij}\le 2c_{l}\,\delta_{ij}\qquad\text{ on
$\Omega_{l}$}
\end{equation}
as bilinear forms, for some constant $c_{l}>0$. By using~\eqref{eq:49}
together with a partition of unity, one can conclude from the
Euclidean case (see, e.g.,~\cite[Th\'eor\`eme~5.5 \&
Th\'eor\`eme~5.7]{MR0227584}) that for $1<q<\infty$, there is a linear
bounded \emph{trace operator}
$T : W^{1,q}(M)\to W^{1-1/q,q}(\partial M)$ with kernel
$\textrm{ker}(T)=W^{1,q}_{0}(M)$ with bounded right inverse
$Z : W^{1-1/q,q}(\partial M)\to W^{1,q}(M)$. For simplicity, we also
write $u_{\vert\partial M}$ for the trace $T(u)$ of $u\in W^{1,q}(M)$
and $\norm{u_{\vert\partial M}}_{q}$ instead of $\norm{T(u)}^q_{L^{q}(\partial
    M)}$.

Similarly, one transfers from the Euclidean case (cf.,
\cite[Th\'eor\`eme~4.2]{MR0227584}) the \emph{Sobolev-trace
    inequality}
\begin{equation}
  \label{eq:68}
  \norm{u_{\vert\partial M}}_{\frac{q(N-1)}{(N-q)}}\lesssim
  \norm{u}_{W^{1,q}(M)},\qquad u\in W^{1,q}(M).
\end{equation}

\subsection{Construction of the Dirichlet-to-Neumann operator}

Let $1<p<\infty$ and $m\ge 0$. Then, by the classical theory of convex minimization  (see,
e.g.,~\cite{MR3369257}), for every boundary
data $\varphi\in W^{1-1/p,p}(\partial M)$, there is a unique \emph{weak
  solution} $u\in (M)$ of the Dirichlet problem~\eqref{eq:19}
(cf.,~\cite{MR3369257}).

\begin{definition}
  For given boundary data $\varphi\in W^{1-1/p,p}$, a function $u\in
  W^{1,p}(M)$ is called a \emph{weak solution} of Dirichlet
  problem~\eqref{eq:19} if $Z\varphi-u\in W^{1,p}_{0}$ and
  \begin{displaymath}
    \int_{M}\abs{\nabla u}^{p-2}_{g}\nabla u\nabla\psi\,+m\,\abs{u}^{p-2}u\psi\td\mu_{g}=0
  \end{displaymath}
  for every $\psi\in C^{\infty}_{c}(M)$.
\end{definition}

Now, we are in the position to define the nonlocal \emph{Dirichlet-to-Neumann
operator} $A$ in $L^2:=L^{2}(\partial M)$ associated with the $p$-Laplace Beltrami
operator $\Delta_{p}$ by
\begin{displaymath}
  A=\left\{\!\! (\varphi,h)\in L^{2}\!\times\! L^2\,\Bigg\vert\!
  \begin{array}[c]{l}
    \exists\,u\in V_{p,2}(M,\partial M)\text{ with trace }
    u_{\vert\partial M}=\varphi\\
    \text{satisfying }\forall\,\psi\in V_{p,2}(M,\partial M):\\
\displaystyle\int_{M}\abs{\nabla u}^{p-2}_{g}\nabla
    u\nabla\psi +m\,\abs{u}^{p-2}u\psi\,\td\mu_{g}=\int_{\partial
    M}\!\! h\,\psi_{\vert\partial M}\,\td\bm{s}_{g}
  \end{array}\!\!\!
\right\}.
\end{displaymath}
In the operator $A$, we denote by $V_{p,2}(M,\partial M)$ the set of
all $u\in W^{1,p}(M)$ with trace
$u_{\vert\partial M}\in L^{2}(\partial M)$. Note, the space
$V_{p,2}(M,\partial M)$ contains the function space
$C^{\infty}(\overline{M})$. It follows from the theory developed
in~\cite{MR3465809} that $A$ is the $T$-sub-differential operator
$\partial_{T}\E$ in $L^2$ (cf.,~\cite{MR3465809}) of the convex,
continuously differentiable, and $T$-elliptic functional
$\E : W^{1,p}(M)\to [0,+\infty)$ defined by
\begin{displaymath}
  \E(u):=\tfrac{1}{p}\displaystyle\int_{M}\left(\abs{\nabla u}^{p}_{g} +m\,\abs{u}^{p}\right)\,\td\mu_{g} 
\end{displaymath}
for every $u\in V_{p,2}(M,\partial M)$. Thus, $A$ is a maximal monotone operator
with dense domain in the Hilbert space $L^2(\partial M)$. One immediately sees
that $A$ is homogeneous of order $\alpha=p-1$.\medskip

Next, suppose $f : \partial M\times\R\to \R$ is a
Lipschitz-continuous \emph{Carath\'eodory} function, that is, $f$
satisfies the following three properties:
\begin{align}
  \label{property:one-Charatheodory}
\bullet\quad & f(\cdot, u) : \partial M\to \R \text{ is measurable on $\partial M$ for every
$u\in \R$,}\\
  \label{property:two-Charatheodory}
\bullet\quad & f(x,0)=0\text{ for a.e. $x\in \partial M$, and}\\
 \notag
\bullet\quad & \text{there is a constant $\omega\ge 0$ such that }\\ \label{eq:2}
    \begin{split}
    \qquad\abs{f(x,u)-f(x,\hat{u})}\le \omega\,\abs{u-\hat{u}}\quad\,
  \text{for all $u$, $\hat{u}\in \R$, a.e. $x\in \partial M$.}
    \end{split}
\end{align}
Then, for every $1\le q\le \infty$,
$F : L^{q}(\partial M)\to L^{q}(\partial M)$
defined by
\begin{displaymath}
  F(u)(x):=f(x,u(x))\qquad\text{ for every $u\in L^{q}(\partial M)$}
\end{displaymath}
is the associated \emph{Nemytskii operator} on
$ L^{q}:=L^{q}(\partial M)$. Moreover, by~\eqref{eq:2}, $F$
is globally Lipschitz continuous on $L^{q}(\partial M)$
with constant $\omega\ge 0$ and $F(0)(x)=0$ for a.e.
$x\in \partial M$.\medskip

Under these assumptions, it follows from
Proposition~\ref{prop:completely-accretive} that the perturbed
operator $A+F$ in $L^{2}(\partial M)$ is an $\omega$-quasi $m$-completely
accretive operator with dense domain $D(A+F)=D(A)$ in $L^{2}(\partial M)$
(see~\cite{MR3369257} or \cite{CoulHau2016} for the details in the Euclidean case). Thus,
$-(A+F)$ generates a strongly continuous semigroup
$\{T_{t}\}_{t\ge 0}$ of Lipschitz-continuous mappings $T_{t}$ on
$L^{2}(\partial M)$ with Lipschitz constant $e^{\omega t}$. For every
$1\le q< \infty$, each $T_{t}$ admits a unique Lipschitz-continu\-ous
extension $T_{t}^{(q)}$ on $L^{q}(\partial M)$ with Lipschitz constant
$e^{\omega t}$ such that $\{T_{t}^{(q)}\}_{t\ge 0}$ is a strongly
continuous semigroup on $L^{q}(\partial M)$, and each $T_{t}^{(q)}$ is
Lipschitz-continu\-ous on
$\overline{L^{2}(\partial M)\cap L^{\infty}(\partial M)}^{\mbox{}_{L^{\infty}}}$ with respect
to the $L^{\infty}$-norm. According to Proposition~\ref{propo:properties-of-St},
for
\begin{displaymath}
\overline{A}^{\mbox{}_{L_{0}}}_{L^{q}(\partial M)}:=\overline{A}^{\mbox{}_{L_{0}}}\cap
(L^{q}(\partial M) \times L^{q}(\partial M))
\end{displaymath}
with $L_{0}:=L_{0}(\partial M,\bm{s}_{g})$, the operator
$-(\overline{A}^{\mbox{}_{L_{0}}}_{L^{q}(\partial M)}+F)$ is the unique infinitesimal
generator of $\{T_{t}^{(q)}\}_{t\ge 0}$ in $L^{q}$. Since
\begin{equation}
  \label{eq:71}
  \overline{A}^{\mbox{}_{L_{0}}}_{L^{q}(\partial M)}u=Au \cap L^{q}(\partial M) 
\end{equation}
for every $u\in D(A)\cap L^{q}(\partial M)$, we call
$\overline{A}^{\mbox{}_{L_{0}}}_{L^{q}(\partial M)}$ the
\emph{realization in $L^{q}(\partial M)$} of the 
\emph{Dirichlet-to-Neumann operator} $A$ associated with the
$p$-Laplace Beltrami operator $\Delta_{p}$. 

For simplicity, we denote
the extension $T_{t}^{(q)}$ on $L^{q}(\partial M)$ of $T_{t}$
again by $T_{t}$.\medskip 

It is worth noting that for $1<p<N$, the semigroup
$\{T_{t}\}_{t\ge 0}$ generated by $-(A+F)$ has an immediate
regularization effect. Indeed, by the
Sobolev-trace inequality~\eqref{eq:68}, the operator $A+F$ satisfies the inequality
\begin{align*}
  \left[u_{\vert\partial M},(A+F)u_{\vert\partial M}\right]_{2}+\omega\,\norm{u_{\vert\partial M}}_{2}^{2}
            &= \int_{\Omega}\abs{\nabla u}_{g}^{p}+m\,\abs{u}^{p}\,\td\mu_{g}\\
           &\hspace{1cm}+\int_{\partial M}f(x,u)u+\omega\,\abs{u}^2\,\td\bm{s}_{g}\\
  &\ge \int_{\Omega}\abs{\nabla
    u}_{g}^{p}+m\,\abs{u}^{p}\,\td\mu_{g}\\
  &\ge \min\{1,m\}\,\norm{u}^{p}_{W^{1,p}(M)}\\
  &\ge C\,\norm{u_{\vert\partial M}}^{p}_{\frac{p(N-1)}{(N-p)}}
\end{align*}
for every $u\in D(A)$, where $[\cdot,\cdot]_{2}$ denotes the duality
brackets on $L^{2}(\partial M)$, and $C>0$ is a constant including
$\min\{1,m\}$ and the constant of the Sobolev-trance inequality. By \cite[Theorem~1.2]{CoulHau2016},
the semigroup $\{T_{t}\}_{t\ge 0}$ satisfies
\begin{equation}
  \label{eq:74}
  \norm{T_{t}u_{0}}_{L^{\frac{p(N-1)}{(N-p)}}(\partial M)}\le
  \left(\tfrac{C}{2}\right)^{\frac{1}{p}}\,t^{\frac{1}{p}}\,e^{\omega
    (\frac{2}{p}+1) t}\norm{u_{0}}_{L^{2}(\partial M)}^{\frac{2}{p}}
\end{equation}
for all $t>0$ and $u_{0}\in L^{2}(\partial M)$. Moreover,
by~\eqref{eq:74} and since
$\{T_{t}\}_{t\ge 0}$ has unique Lipschitz-continuous extension on
$L^{1}(\partial M)$, the same theorem infers that
the semigroup $\{T_{t}\}_{t\ge 0}$ satisfies for every $1\le q\le
(N-1)\, q_{0}/(N-p)$ satisfying $q>(2-p)(N-1)/(p-1)$ the following
\emph{$L^{q}$-$L^{\infty}$-regularity estimate}
\begin{equation}
  \label{eq:Lq-Linfty-estimates-diff}
  \norm{T_{t}u_{0}}_{L^{\infty}(\partial M)}\lesssim\,t^{-\alpha_{q}}\,e^{\omega\beta_{q}
    t} \,\norm{u_{0}}_{L^{q}(\partial M)}^{\gamma_{q}}
\end{equation}
for every $t>0$, $u_{0}\in L^{q}(\partial M)$, with exponents
\begin{displaymath}
  \begin{split}
    &
    \alpha_{q}=\frac{\alpha^{\ast}}{1-\gamma^{\ast}\left(1-\frac{q(N-p)}{(N-1)
          q_{0}} \right)},\qquad
    \beta_{q}=\frac{\frac{\beta^{\ast}}{2}+\gamma^{\ast}
      \frac{q(N-p)}{(d-1)q_{0}}}{1-\gamma^{\ast}\left(1-\frac{q(N-p)}{(N-1)
          q_{0}}\right)},\\
    & \mbox{}\qquad\quad \gamma_{q}=\frac{\gamma^{\ast}\,q(N-p)}{(N-1)
      q_{0}\left(1-\gamma^{\ast}\left(1-\frac{q (N-p)}{(N-1)
            q_{0}}\right)\right)},
  \end{split}
\end{displaymath}
where $q_{0}\ge p$ is chosen
(minimal) such that $\left(\frac{N-1}{N-p}-1\right)q_{0}+p-2>0$ and
\begin{displaymath}
  \begin{split}
    & \alpha^{\ast}:=\frac{N-p}{(p-1)\, q_{0}+(N-p)(p-2)},\quad
    \beta^{\ast}:=\frac{(\frac{2}{p}-1)N+p-\frac{2}{p}}{(p-1)
      q_{0}+(N-p)(p-2)}+1,\\
    &\mbox{}\hspace{3cm}\gamma^{\ast}:=\frac{(p-1)\, q_{0}}{(p-1)\,
      q_{0}+(N-p)(p-2)}.
  \end{split}
\end{displaymath}


 \subsection{Global regularity estimates on $\frac{\td u}{\dt}$}
\label{subsec:global-regularity-estimates}

Throughout this subsection, let $p\in
(1,\infty)\setminus\{2\}$. Since, the Dirichlet-to-Neumann operator
$A$ in $L^{2}(\partial M)$ is homogeneous of order $\alpha=p-1$,
identity~\eqref{eq:71} yields that for every $1\le q<\infty$, the
realization $\overline{A}^{\mbox{}_{L_{0}}}_{L^{q}(\partial M)}$ of
$A$ in $L^{q}(\partial M)$ is also homogeneous of order $p-1$. Thus,
by Corollary~\ref{cor:RN-Lipschitz-case}, for every $1<q<\infty$ and
$u_{0}\in L^{q}(\partial M)$, the function $u(t)=T_{t}u_{0}$
is differentiable a.e. on $(0,\infty)$ and
satisfies
\begin{align*}
    \lnorm{\frac{\td T_{t}u_{0}}{\dt}_{\!\! +}}_{L^{q}(\partial M)}
    &\le \frac{\norm{u_{0}}_{L^{q}(\partial M)}}{\abs{p-2}\,t}\left[1+e^{\omega
    t} +\omega \int_{0}^{t}\left(1+e^{\omega
    s} \right)\,e^{\omega(t-s)}\ds\right]
\end{align*}
  for every $t>0$. Note, 
  the right hand side of this estimate can be rearranged as follows
\begin{equation}
  \label{eq:48Lq}
    \lnorm{\frac{\td T_{t}u_{0}}{\dt}_{\!\! +}}_{L^{q}(\partial M)}
  \le \frac{\left[2 +\omega\, t\right]\,e^{\omega t}}{\abs{p-2}\,t}\norm{u_{0}}_{L^{q}(\partial M)}
\end{equation}
for every $t>0$. Since the boundary $\partial M$ is compact,
H\"older's inequality gives
\begin{displaymath}
  \lnorm{\frac{\td T_{t}u_{0}}{\dt}_{\!\! +}}_{L^{1}(\partial M)}\le
  \bm{s}_{g}^{1/q'}(\partial M) \frac{\left[2 +\omega\, t\right]\,e^{\omega t}}{\abs{p-2}\,t}\norm{u_{0}}_{L^{q}(\partial M)}
\end{displaymath}
for every $q>1$ and hence, if we fix $u_{0}\in L^{2}(\partial M)$,
then sending $q\to 1+$ in the above inequality shows
that~\eqref{eq:48Lq}, in particular, holds for $q=1$. By~\eqref{eq:Lq-Linfty-estimates-diff}, for either
$p>2$ or $(2N-1)/N<p<2$, one has that  $u_{0}\in L^{1}(\partial M)$
yields that $T_{t}u_{0}\in
L^{\infty}(\partial M)\hookrightarrow L^{2}(\partial M)$ for all
$t>0$. Hence, for this range of $p$, we get that \eqref{eq:48Lq} hold
for $q=1$ and $u_{0}\in L^{1}(\partial M)$.

Next, let
  $u_{0}\in L^{\infty}(\partial M)$ and $t>0$. We assume
  $\norm{\frac{\td T_{t}u_{0}}{\dt}_{\!\! +}}_{L^{\infty}(\partial M)}>0$ (otherwise,
  there is nothing to show). Then, for every
  $s\in (0,\norm{\frac{\td T_{t}u_{0}}{\dt}_{\!\! +}}_{L^{\infty}(\partial M)})$ and
  $2\le q<\infty$, Chebyshev's inequality yields
  \begin{displaymath}
    \bm{s}_{g}\left(\Bigg\{\labs{\frac{\td T_{t}u_{0}}{\dt}_{\!\! +}}\ge s
        \Bigg\}\right)^{1/q}\le \frac{\lnorm{\frac{\td T_{t}u_{0}}{\dt}_{\!\! +}}_{L^{q}(\partial M)}}{s}
  \end{displaymath}
  and so, by~\eqref{eq:48Lq},
  \begin{displaymath}
    s\,\bm{s}_{g}\left(\Bigg\{\labs{\frac{\td T_{t}u_{0}}{\dt}}\ge s
        \Bigg\}\right)^{1/q}\le 
\frac{\left[2 +\omega\, t\right]\,e^{\omega t}}{\abs{p-2}\,t}\norm{u_{0}}_{L^{q}(\partial M)}
  \end{displaymath}
 Thus and since $\lim_{q\to
    \infty}\norm{u_{0}}_{L^{q}(\partial
    M)}=\norm{u_{0}}_{L^{\infty}(\partial M)}$, sending
  $q\to+\infty$ in the last inequality, yields
\begin{displaymath}
    s \le \frac{\left[2 +\omega\, t\right]\,e^{\omega t}}{\abs{p-2}\,t}\norm{u_{0}}_{L^{\infty}(\partial M)}
  \end{displaymath}
  and since
  $s\in (0,\lnorm{\frac{\td T_{t}u_{0}}{\dt}_{\!\!  +}}_{L^{\infty}(\partial M)})$ was
  arbitrary, we have thereby shown that~\eqref{eq:48Lq} also holds for
  $q=\infty$.\medskip


  Finally, for $p\in (1,N)\setminus\{2\}$, we can apply
  Corollary~\ref{cor:1} or, alternatively,
  combine~\eqref{eq:Lq-Linfty-estimates-diff} with \eqref{eq:48Lq} for
  $q=\infty$. Then, we find that
\begin{equation}
  \label{eq:62}
  \lnorm{\frac{\td T_{t}u_{0}}{\dt}_{\!\! +}}_{L^{\infty}(\partial M)}
  \lesssim \frac{2\,\left[2 +\tfrac{\omega}{2}\, t\right]\,e^{\omega\,(1+\frac{\beta_{q}}{2})
    t}}{\abs{p-2}\,t^{\alpha_{q}+1}}\,\norm{u_{0}}_{L^{q}(\partial M)}^{\gamma_{q}}
\end{equation}
for every $t>0$, $u_{0}\in L^{q}(\partial M)$, and $1\le q\le
(N-1)\, q_{0}/(N-p)$ satisfying $q>(2-p)(N-1)/(p-1)$.

By this computation together with Theorem~\ref{thm:2}, we can
  state the following regularity result on mild solutions to the
  elliptic-parabolic problem~\eqref{eq:12}.

  \begin{theorem}\label{thm:reg-DtN}
    Let $N\ge 2$ and $1<p<\infty$. Then every mild solution $u$ of the elliptic-parabolic
    problem~\eqref{eq:12} admits the following additional regularity.     
    \begin{enumerate}
      \item ({\rm $L^1$ Aronson-B\'enilan type estimates}) If either $(2N-1)/N<p<2$ or $p>2$, then for every $1\le
        q\le \infty$ and $u_{0}\in L^{q}(\partial\Omega)$, the mild solution 
        $u(t):=T_{t}u_{0}$ of the elliptic-parabolic
      problem~\eqref{eq:12} is differentiable for a.e. $t>0$, is a
      \emph{strong solution} in $L^{q}(\partial\Omega)$ of~\eqref{eq:12}, and satisfies
      \begin{displaymath}
        \lnorm{\frac{\td u}{\dt}_{\!\! +}\!\!(t)}_{L^{q}(\partial M)}
        \le \frac{\left[2 +\omega\, t\right]\,e^{\omega
            t}}{\abs{p-2}\,t}\norm{u_{0}}_{L^{q}(\partial M)}
        \qquad\text{for every $t>0$.}
      \end{displaymath}
      
    \item ({\rm Extrapolated $L^1$ Aronson-B\'enilan type estimates}) Let $p\in (1,N)\setminus\{2\}$. Then, in addition to
      statement~(1), for every
      $1\le q\le (N-1)\, q_{0}/(N-p)$ satisfying $q>(2-p)(N-1)/(p-1)$
      and $u_{0}\in L^{q}(\partial\Omega)$, the mild solution
      $u(t):=T_{t}u_{0}$ of the elliptic-parabolic
      problem~\eqref{eq:12} satisfies
      \begin{displaymath}
        \lnorm{\frac{\td u}{\dt}_{\!\! +}\!\! (t)}_{L^{\infty}(\partial M)}
        \lesssim \frac{2\,\left[2 +\tfrac{\omega}{2}\, t\right]\,e^{\omega\,(1+\frac{\beta_{q}}{2})
            t}}{\abs{p-2}\,t^{\alpha_{q}+1}}\,\norm{u_{0}}_{L^{q}(\partial M)}^{\gamma_{q}}\qquad\text{for every $t>0$.}
      \end{displaymath}

    \item ({\rm Point-wise Aronson-B\'enilan type estimates}) If
      either $(2N-1)/N<p<2$ or $p>2$, then for every $1\le q\le
      \infty$  and positive $u_{0}\in L^{q}(\partial\Omega)$, the strong
      solution $u$ of problem~\eqref{eq:12} satisfies
      \begin{displaymath}
        (p-2)\frac{\td u}{\dt}_{\!\!+}\!\!(t)\ge -\frac{u(t)}{t}+(p-2)\,g_{0}(t),
      \end{displaymath}
      for a.e. $t>0$, where $g_{0} : (0,\infty)\to L^q (\partial\Omega)$ is a
      measurable function.
    \end{enumerate}
  \end{theorem}

%
%


\providecommand{\bysame}{\leavevmode\hbox to3em{\hrulefill}\thinspace}

\end{document}